\documentclass[10pt]{amsproc}
\usepackage{amssymb,amsthm,stmaryrd,amsmath,amscd,amsrefs,bbm,todonotes}
\usepackage{tikz-cd}
\usepackage[hidelinks]{hyperref}
\usepackage{chngcntr}
\usepackage[foot]{amsaddr}

\counterwithin{equation}{subsection}
\newtheorem{theorem}[equation]{Theorem}
\newtheorem{lemma}[equation]{Lemma}
\newtheorem{proposition}[equation]{Proposition}
\newtheorem{corollary}[equation]{Corollary}
\newtheorem{theoremA}{Theorem}

\theoremstyle{definition}
\newtheorem*{question}{Question}
\newtheorem{notation}[equation]{Notation}

\newtheorem{construction}[equation]{Construction}
\newtheorem{remark}[equation]{Remark}

\newtheorem{definition}[equation]{Definition}

\title[Potentially diagonalisable lifts]{Potentially diagonalisable lifts with controlled Hodge--Tate weights}
\author{Robin Bartlett}
\address{Max Planck Institute for Mathematics, Bonn}
\email{robinbartlett18@mpim-bonn.mpg.de}
\begin{document}
\maketitle
	
\begin{abstract}
	Motivated by the weight part of Serre's conjecture we consider the following question. Let $K/\mathbb{Q}_p$ be a finite extension and suppose $\overline{\rho} \colon G_K \rightarrow \operatorname{GL}_n(\overline{\mathbb{F}}_p)$ admits a crystalline lift with Hodge--Tate weights contained in the range $[0,p]$. Does $\overline{\rho}$ admits a potentially diagonalisable crystalline lift of the same Hodge--Tate weights? We answer this question in the affirmative when $K = \mathbb{Q}_p$ and $n \leq 5$, and $\overline{\rho}$ satisfies a mild `cyclotomic-free' condition. We also prove partial results when $K/\mathbb{Q}_p$ is unramified and $n$ is arbitrary.
\end{abstract}
\tableofcontents

\section{Introduction}

Let $p$ be a prime and $K/\mathbb{Q}_p$ a finite extension. A potentially diagonalisable $p$-adic representation of $G_K$ is one which is potentially crystalline and which, possibly after restriction to $G_{K'}$ for an extension $K'/K$, lies in the same irreducible component of a potentially crystalline deformation ring as a representation which is a sum of crystalline characters. The notion was introduced in \cite{BLGGT} where very general change of weight theorems for automorphic Galois representations were proven under the assumption that the representations in question are potentially diagonalisable above $p$, cf. \cite[Theorem E]{BLGGT}.

One motivation for these theorems comes from the following question, which is a generalisation of Serre's classical modularity conjecture. If $F$ is a CM field and $\overline{r}\colon G_F \rightarrow \operatorname{GL}_n(\overline{\mathbb{F}}_p)$ is continuous and irreducible, then what are the possible weights for which there exists an automorphic representation giving rise to $\overline{r}$? Very little is currently known about this question. However if one assumes that $\overline{r}$ is automorphic of \emph{some} weight then the outlook is better---in this case if $\overline{r}_v$ admits a potentially diagonalisable crystalline lift at each place $v \mid p$ of $F$ then, under a Taylor--Wiles hypothesis, it can be shown that $\overline{r}$ is automorphic with weight equal to the Hodge--Tate weights of the lifts of $\overline{r}_v$, cf. \cite[Theorem 3.1.3]{BLGG} or \cite[Theorem 5.2.1]{BG18}. Since representations associated to automorphic forms which are unramified above $p$ are crystalline above $p$, one is lead to the following (purely local) question:

\begin{question}
Let $K/\mathbb{Q}_p$ be a finite extension and let $\overline{\rho}\colon G_K \rightarrow \operatorname{GL}_n(\overline{\mathbb{F}}_p)$ be a continuous representation. If $\overline{\rho}$ has a crystalline lift\footnote{By a crystalline lift we mean a crystalline representation $\rho\colon G_K \rightarrow \operatorname{GL}_n(\overline{\mathbb{Z}}_p)$ such that $\rho \otimes_{\overline{\mathbb{Z}}_p} \overline{\mathbb{F}}_p \cong \overline{\rho}$} then does $\overline{\rho}$ have a potentially diagonalisable crystalline lift with the same Hodge--Tate weights?
\end{question} 

When the answer to this question is yes one can deduce cases of the weight part of Serre's modularity conjecture. In particular the results of this paper can be applied in this way. However, since our focus here is local, we do not express these applications explicitly.

When $K/\mathbb{Q}_p$ is unramified Hui Gao and Tong Liu \cite{GT14} have shown that any crystalline representation with Hodge--Tate weights contained in $[0,p-1]$ is potentially diagonalisable, so in this case the answer to the question is yes. When $n=2$ and the weights are contained in $[0,p]$ the answer is also known to be yes by work of Toby Gee, Tong Liu and David Savitt \cite{GLS,GLS15} (see also the work of Xiyuan Wang \cite{Wang17} who extended their proof to the case $p=2$). These results were used to establish the weight part of Serre's conjecture in dimension two. When $K/\mathbb{Q}_p$ is unramified, the author \cite{B18} has shown the answer is yes for weights in $[0,p]$ and semi-simple $\overline{\rho}$ of any dimension.

In this paper we address the question for weights in the range $[0,p]$, in the presence of non-split extensions. Our methods are generalisations of those employed in \cite{GLS,GLS15} to dimensions $>2$. Such generalisations have already been made by Hui Gao \cite{Gao17a,Gao17b} in the case when every Jordan--Holder factor of $\overline{\rho}$ is one-dimensional (and under some additional conditions). The main innovation in this paper is to put the calculations used in the works mentioned above in a more conceptual framework. This allows us to consider representations whose Jordan--Holder factors are not one-dimensional, and to remove some of the conditions appearing in Gao's work.

One crucial assumption that is necessary for our methods is that $\overline{\rho}$ be cyclotomic-free (cf. Definition~\ref{cycfreecondition}). This is an $n$-dimensional generalisations of the avoidance of representations of the shape $(\begin{smallmatrix}
\chi_{\operatorname{cyc}} & * \\ 0 & 1
\end{smallmatrix})$. Our first theorem is then the following.

\begin{theoremA}\label{thma}
	Let $\overline{\rho}\colon G_K \rightarrow \operatorname{GL}_n(\overline{\mathbb{F}})$ be continuous and cyclotomic-free. Suppose  there exists a crystalline representation $\rho \colon G_K \rightarrow \operatorname{GL}_n(\overline{\mathbb{Z}}_p)$ with $\overline{\rho} \cong \rho \otimes_{\overline{\mathbb{Z}}_p} \overline{\mathbb{F}}_p$ and with Hodge--Tate weights $\in [0,p]$. 
	
	If $K = \mathbb{Q}_p$ and $n \leq 5$ then there exists a potentially diagonalisable crystalline representation $\rho'$ with $\overline{\rho} \cong \rho' \otimes_{\overline{\mathbb{Z}}_p} \overline{\mathbb{F}}_p$ and with Hodge--Tate weights equal to those of $\rho$.
\end{theoremA}

In particular the question has the answer yes for $\overline{\rho}$ as in Theorem~\ref{thma}. As we shall explain below, the assumptions $K = \mathbb{Q}_p$ and $n \leq 5$ ensure that certain irreducible semilinear objects have a particularly simple form and admit crystalline lifts. Beyond these low dimensional cases the situation is more complicated; this is what prevents us from answering the question in greater generality. However, when the Jordan--Holder factors of $\overline{r}$ are one-dimensional this issue does not arise and we are also able to prove:
\begin{theoremA}\label{thmB}
	Suppose $K/\mathbb{Q}_p$ is unramified. Let $\overline{\rho}\colon G_K \rightarrow \operatorname{GL}_n(\overline{\mathbb{F}})$ be continuous and cyclotomic-free. Suppose there exists a crystalline representation $\rho \colon G_K \rightarrow \operatorname{GL}_n(\overline{\mathbb{Z}}_p)$ with $\overline{\rho} \cong \rho \otimes_{\overline{\mathbb{Z}}_p} \overline{\mathbb{F}}_p$ and with $\tau$-Hodge--Tate weights in $[0,p]$ for each $\tau \in \operatorname{Hom}_{\mathbb{F}_p}(k,\overline{\mathbb{F}}_p)$. 
	
	If every Jordan--Holder factor of $\overline{\rho}$ is one-dimensional then there exists a potentially diagonalisable crystalline representation $\rho'$ with $\overline{\rho} \cong \rho' \otimes_{\overline{\mathbb{Z}}_p} \overline{\mathbb{F}}_p$ and with $\tau$-Hodge--Tate weights equal to those of $\rho$ for every $\tau \in \operatorname{Hom}_{\mathbb{F}_p}(k,\overline{\mathbb{F}}_p)$.
\end{theoremA}

This gives a more general version of the main results of \cite{Gao17a,Gao17b}.

We now explain how potentially diagonalisable lifts may be produced. Every mod~$p$ representation is a successive extension of irreducible representations, each of which is induced  over an unramified extension of $K$ from a character. Thus the standard method for producing potentially diagonalisable lifts is to consider lifts obtained in the same way, by taking successive crystalline extensions of irreducible representations obtained by inducing crystalline characters. Such lifts are called obvious lifts (following terminology introduced in \cite[Subsection 7.1]{GHS}) and it is straightforward to see that obvious lifts are potentially diagonalisable (see the beginning of Section~\ref{section7}). In both Theorem~\ref{thma} and~\ref{thmB} the $\rho'$ constructed will be obvious lifts.

There are two issues with this method of producing obvious lifts. Firstly, it is probably not true that every mod~$p$ representation has an obvious lift. Recent work of Matthew Emerton and Toby Gee shows, using geometric methods, that potentially diagonalisable lifts can always be obtained, but their methods involve an inductive process where obvious lifts are allowed to vary inside irreducible components of deformation rings. We avoid this problem by restricting attention to cyclotomic-free representations. Secondly, even if one can produce obvious lifts, the Hodge--Tate weights of such lifts seem to be very restrictive. They depend upon $\overline{\rho}$ in a non-obvious way, and it is unclear that if $\overline{\rho}$ has a crystalline lift of some weight then it will be possible to construct an obvious lift of the same weight.

In this paper we resolve the second issue, not by producing obvious crystalline lifts of $\overline{\rho}$, but instead by producing lifts of a semilinear object related to $\overline{\rho}$. More precisely, if $\overline{\rho}$ admits a crystalline lift then Kisin's work in integral $p$-adic Hodge theory \cite{Kis06} associates to this lift a Breuil--Kisin module. We produce obvious crystalline lifts of the mod~$p$ reduction of this Breuil--Kisin module (i.e. we produce obvious lifts whose associated Breuil--Kisin module is congruent modulo~$p$ to the one arising from the previous crystalline lift). Our assumption that $\overline{\rho}$ is cyclotomic-free means that this obvious lift will also be an obvious lift of $\overline{\rho}$. The key ingredient which makes this possible is a theorem of Gee--Liu--Savitt which says that the Breuil--Kisin module of a crystalline representation with Hodge--Tate weights contained in $[0,p]$ is of a particularly nice form; in particular it's reduction modulo~$p$ sees the Hodge--Tate weights of the crystalline representation it was obtained from.

We conclude our introduction by explaining the content of this paper. Section~\ref{section1} discusses the cyclotomic-freeness condition and its consequences. The main result is that, if $K_\infty$ is the extension of $K$ obtained by adjoining a compatible system of $p$-th power roots of a uniformiser of $K$, then any $G_{K_\infty}$-equivariant morphism between cyclotomic-free $G_K$-representations is $G_K$-equivariant.

In Section~\ref{section3} we recall the notion of a Breuil--Kisin module. We state Kisin's construction which associates a Breuil--Kisin module to a crystalline representation and the result of Gee--Liu--Savitt which controls the shape of such Breuil--Kisin modules. We also recall from \cite{B18} the notion of strong divisibility for $p$-torsion Breuil--Kisin modules, and some of the properties such modules satisfy.

In Section~\ref{section4} we compute the space of extensions of strongly divisible Breuil--Kisin modules. The dimension is closely related to the dimension of crystalline extensions in characteristic zero, and in Section~\ref{section5} we use this to show that extensions between strongly divisible Breuil--Kisin modules admit lifts by crystalline extensions.

The previous two sections reduce the problem of lifting strongly divisible Breuil--Kisin modules to that of lifting irreducible such modules. In general the structure of such modules is complicated, and we do not know how to produce such lifts. However we show in Section~\ref{section6}, by explicit computation, that in low dimensional situations (when $K = \mathbb{Q}_p$ and $n \leq 4$) their structure can be controlled so that crystalline lifts can be produced. In the final section we put these results together to prove the theorems stated above.

\subsection*{Acknowledgements}

Part of this work was done during my PhD and I would like to thank my advisor Fred Diamond for his guidance and support. I would also like to thank Matthew Bisatt for helpful conversations.

This work was partly supported by the Engineering and Physical Sciences Research Council
[EP/L015234/1] and the EPSRC Centre for Doctoral Training in Geometry and Number
Theory (The London School of Geometry and Number Theory), University College London.
\section{Cyclotomic-free representations}\label{section1}

For this section let $K/\mathbb{Q}_p$ be any finite extension. Let $K_\infty = K(\pi^{1/p^\infty})$ where $\pi^{1/p^\infty}$ is a fixed choice of compatible system of $p$-th power roots of a uniformiser $\pi \in K$. Our aim is to understand restriction from $G_K$ to $G_{K_\infty}$ for a class of representations we call cyclotomic-free. 

If $G$ is a topological group denote by $\operatorname{Rep}(G)$ the category of continuous representations of $G$ on finite dimensional $\overline{\mathbb{F}}_p$-vector spaces. In this section all unadorned tensor products are over $\overline{\mathbb{F}}_p$.

\subsection{Cyclotomic-free representations} 
	For any field $\mathbb{F}$ of characteristic $p$ let $\mathbb{F}(1)$ denote the $G_K$-representation whose underlying vector space is $\mathbb{F}$, with $G_K$-action given by the mod~$p$ cyclotomic character $\chi_{\operatorname{cyc}}$.
\begin{definition}\label{cycfreecondition}
	$V \in \operatorname{Rep}(G_K)$ is cyclotomic-free if $V$ admits a composition series $0 = V_n \subset \ldots \subset V_0 =V$ such that $V_i/V_{i+1} \otimes \overline{\mathbb{F}}_p(1)$ is not a Jordan--Holder factor of $V_{i+1}$ for any $i$. Pictorially
	$$
	V \sim 
	\begin{pmatrix}
	\ddots & * & * \\   0 & V_1/V_2 & * \\  0 & 0 & V_0/V_1
	\end{pmatrix}
	$$
	and we ask that for each $i$ no block above $V_i/V_{i+1}$ is isomorphic to $V_i/V_{i+1} \otimes \overline{\mathbb{F}}_p(1)$. In particular cyclotomic-freeness is ruling out representations of the form $\big(\begin{smallmatrix}
	\chi_{\operatorname{cyc}} & * \\ 0 & 1
	\end{smallmatrix} \big)$.
\end{definition}

Note in the definition of cyclotomic-freeness we require that one composition series of $V$ satisfies the conditions describes in \eqref{cycfreecondition}, not that every one does.
\begin{lemma}\label{quotient}
	If $V \in \operatorname{Rep}(G_K)$ is cyclotomic-free then any subquotient of $V$ is cyclotomic-free also.
\end{lemma}
\begin{proof}
	Suppose $f\colon V \rightarrow W$ is surjective and let $(V_i)_i$ be a composition series as in Definition~\ref{cycfreecondition}. Set $W_i = f(W_i)$. Then $f$ induces surjective maps $V_i/V_{i+1} \rightarrow W_i/W_{i+1}$. Thus $W_i/W_{i+1} \cong V_i/V_{i+1}$ or $W_i/W_{i+1} =0$. Thus, after re-indexing, the $(W_i)_i$ form a composition series as in Definition~\ref{cycfreecondition}. On the other hand, if $W \subset V$ is a $G_K$-stable subspace set $W_i = W \cap V_i$. Then $W_i/W_{i+1} \hookrightarrow V_i/V_{i+1}$ and so either $W_i/W_{i+1} \cong V_i/V_{i+1}$ or is zero. Again after re-indexing the $(W_i)_i$ form a composition series as in Definition~\ref{cycfreecondition}.
\end{proof}
On the other hand the class of cyclotomic-free representations is not closed under extensions; if the mod~$p$ cyclotomic character is trivial it is not even closed under direct sums.

The main result of this section is:
\begin{theorem}\label{ff}
	Let $V$ and $W$ be cyclotomic-free $G_K$-representations and $f\colon V \rightarrow W$ a morphism of $G_{K_\infty}$-representations. Then $f$ is $G_K$-equivariant.
\end{theorem}

We give a proof over the next three subsections.
\subsection{Restriction for irreducibles}\label{tame} Let $K^{\operatorname{t}}$ be the maximal tamely ramified extension of $K$. Since $K_\infty$ is totally wildly ramified $K_\infty \cap K^{\operatorname{t}} = K$. Galois theory then tells us the restriction map
$$
 \operatorname{Gal}(K^{\operatorname{t}}_\infty/ K_\infty) \rightarrow \operatorname{Gal}(K^{\operatorname{t}}/K)
$$
where $K_\infty^{\operatorname{t}} = K_\infty K^{\operatorname{t}}$, is an isomorphism. By \cite[Proposition 4]{Serre72} the action of $G_K$ on any semi-simple object of $\operatorname{Rep}(G_K)$ factors through $\operatorname{Gal}(K^{\operatorname{t}}/K)$. Likewise the $G_{K_\infty}$-action on any semi-simple object of $\operatorname{Rep}(G_{K_\infty})$ factors through $\operatorname{Gal}(K_\infty^{\operatorname{t}}/K_\infty)$. This proves:
\begin{lemma}\label{equiv}
	Restriction induces an equivalence between the category of semi-simple $V \in \operatorname{Rep}(G_K)$ and the category of semi-simple $V \in \operatorname{Rep}(G_{K_\infty})$.
\end{lemma}

This implies Theorem~\ref{ff} holds when $V$ and $W$ are irreducible. It also implies that any $G_K$-composition series of $V \in \operatorname{Rep}(G_K)$ is also a $G_{K_\infty}$-composition series. In particular, if $V$ is cyclotomic-free in the sense of Definition~\ref{cycfreecondition} then $V|_{G_{K_\infty}}$ is cyclotomic-free in the following sense.

\begin{definition}[Cyclotomic-freeness for $G_{K_\infty}$-representations]\label{cycfreeGKinfty}
	$V \in \operatorname{Rep}(G_{K_\infty})$ is cyclotomic-free if $V$ admits a composition series $0 = V_n \subset \ldots \subset V_0 =V$ such that $V_i/V_{i+1} \otimes \overline{\mathbb{F}}_p(1)$ is not a Jordan--Holder factor of $V_{i+1}$.
\end{definition}

\subsection{Restriction on Galois cohomology} The following lemma is well-known. See for example \cite[Lemma 2.1.2]{B18}.
\begin{lemma}\label{irredinduced}
	Every irreducible $V \in \operatorname{Rep}(G_K)$ is isomorphic to $\operatorname{Ind}_L^K \chi$ where $L/K$ is an unramified extension and $\chi:G_L \rightarrow \overline{\mathbb{F}}_p^\times$ is a continuous character.
\end{lemma}

We have written $\operatorname{Ind}_L^K$ in place of $\operatorname{Ind}_{G_L}^{G_K}$ and we continue with this notation throughout. 

If $G$ is a topological group and $V,W \in \operatorname{Rep}(G)$ write $\operatorname{Hom}(V,W) \in\operatorname{Rep}(G)$ for the representation with underlying vector space $\operatorname{Hom}_{\overline{\mathbb{F}}_p}(V,W)$ and $G$-action given by $\sigma \cdot f = \sigma \circ f \circ \sigma^{-1}$. Note that $\operatorname{Hom}(V,W) = V^\vee \otimes W$ where $V^\vee = \operatorname{Hom}(V,\overline{\mathbb{F}}_p)$. We shall use that if $W \in \operatorname{Rep}(G_L)$ and $V \in \operatorname{Rep}(G_K)$ then there are isomorphisms $\operatorname{Ind}_L^K( V|_L \otimes W) \cong V \otimes \operatorname{Ind}_L^K W$.

If $G$ is profinite we let $H^*(G,V)$ denote the continuous cohomology groups valued in $V$.
\begin{lemma}\label{injective}
	If $V,W \in \operatorname{Rep}(G_K)$ are irreducible then the restriction map
	\begin{equation}\label{restriction}
	H^1(G_K,\operatorname{Hom}(V,W)) \rightarrow H^1(G_{K_\infty},\operatorname{Hom}(V,W))
	\end{equation}
	is injective unless $W \cong V \otimes \overline{\mathbb{F}}_p(1)$.
	
\end{lemma}
\begin{proof}
	First suppose that $\operatorname{Hom}(V,W)$ in \eqref{restriction} is replaced by $\operatorname{Ind}_L^K Z$ with $Z$ one dimensional. Recall that $\operatorname{Ind}_L^K Z$ is the vector space of continuous functions $f\colon G_K \rightarrow Z$ such that $f(hg) = h \cdot f(g)$ for all $h \in G_L, g \in G_K$. Since $G_{L_\infty} = G_{K_\infty} \cap G_L$, restriction of functions describes a map $\operatorname{Ind}_L^K Z \rightarrow \operatorname{Ind}_{L_\infty}^{K_\infty} Z$ fitting into the $G_{L_\infty}$-equivariant diagram
	$$
	\begin{tikzcd}[column sep =tiny, row sep=small]
	\operatorname{Ind}_L^K Z \arrow{r} \arrow{d} & \operatorname{Ind}_{L_\infty}^{K_\infty} Z  \arrow{d} \\
	Z \arrow{r} & Z|_{L_\infty}
	\end{tikzcd}
	$$
	The vertical arrows are evaluation at $1$. In cohomology, the vertical arrows induce isomorphisms and the horizontal arrows induce restriction, cf. \cite[Section 2.5]{Serre}. By \cite[Lemma 5.4.2]{GLS15} $H^1(G_L,Z) \rightarrow H^1(G_{L_\infty},Z)$ is injective unless $Z \cong \overline{\mathbb{F}}_p(1)$, and so the claim is also true for the restriction map $H^1(G_K,\operatorname{Ind}_L^KZ) \rightarrow H^1(G_{K_\infty},\operatorname{Ind}_L^K Z)$.
	 
	 Return to the statement of the lemma. As $W$ is irreducible $W \cong \operatorname{Ind}_L^K Z$ for a one-dimensional $Z$. The projection formula gives $\operatorname{Hom}(V,W) \cong  \operatorname{Ind}_L^K(  Z \otimes V^\vee|_L)$. Since $V$ is irreducible $V^\vee$ is irreducible also, and so $V^\vee \cong \operatorname{Ind}_F^K Y$ for a one-dimensional $Y$. Mackey's theorem implies $V^\vee|_L \cong \bigoplus_\gamma \operatorname{Ind}_{FL}^L Y^{(\gamma)}$ with $\gamma$ running over some subset of $G_L$ and where $Y^{(\gamma)} = Y$ as vector spaces with $G_{FL}$-action given by $\sigma(y) = (\gamma \sigma \gamma^{-1})(y)$. Therefore
	\begin{equation}\label{Hom=}
	\operatorname{Hom}(V,W) \cong \bigoplus_\gamma \operatorname{Ind}_L^K(Z \otimes \operatorname{Ind}_{FL}^L Y^{(\gamma)}) = \bigoplus_\gamma \operatorname{Ind}_{FL}^{K} (Z \otimes Y^{(\gamma)})
	\end{equation}
	Thus \eqref{restriction} is a direct sum of restriction maps $H^1(G_K,\operatorname{Ind}_{FL}^K (Z \otimes Y^{(\gamma)})) \rightarrow H^1(G_{K_\infty},\operatorname{Ind}_{FL}^K (Z \otimes Y^{(\gamma)}))$. By the previous paragraph, these maps are injective unless $Z \otimes Y^{(\gamma)} \cong \overline{\mathbb{F}}_p(1)$. However if this is the case then \eqref{Hom=} implies $\operatorname{Hom}(V\otimes \overline{\mathbb{F}}_p(1),W)$ has $G_K$-fixed points, i.e. $W \cong V \otimes \overline{\mathbb{F}}_p(1)$. 
\end{proof}
\begin{lemma}\label{devissage}
	Suppose $W$ is cyclotomic-free, $V$ is irreducible, and that $V \otimes \overline{\mathbb{F}}_p(1)$ is not a Jordan--Holder factor of $W$. Then
	$$
	H^i(G_K,\operatorname{Hom}(V,W)) \rightarrow H^i(G_{K_\infty},\operatorname{Hom}(V,W))
	$$
	is an isomorphism if $i =0$ and is injective if $i =1$.
\end{lemma}
\begin{proof}
	Induct on the length of $W$. If $W$ is irreducible then the claim follows from Lemma~\ref{equiv} and Lemma~\ref{injective}. If $W$ is not irreducible we can fit $W$ into an exact sequence $0 \rightarrow W_1 \rightarrow W \rightarrow W_2 \rightarrow 0$ where $W_1$ is cyclotomic-free, $W_2$ is irreducible and non-zero, and where $W_2 \otimes \overline{\mathbb{F}}_p(1)$ not a Jordan--Holder factor of $W_1$. Passing to cohomology gives the diagram
		$$
		\hspace*{-1cm}
		\begin{tikzcd}[column sep =tiny, row sep=small]
		H^0(G_K,\operatorname{Hom}(V,W_2)) \arrow{r} \arrow{d} & H^1(G_K,\operatorname{Hom}(V,W_1)) \arrow{r} \arrow{d} & H^1(G_K,\operatorname{Hom}(V,W)) \arrow{r} \arrow{d} & H^1(G_K,\operatorname{Hom}(V,W_2)) \arrow{d} \\
		H^0(G_{K_\infty},\operatorname{Hom}(V,W_2)) \arrow{r}  &H^1(G_{K_\infty},\operatorname{Hom}(V,W_1)) \arrow{r}  & H^1(G_{K_\infty},\operatorname{Hom}(V,W)) \arrow{r}  & H^1(G_{K_\infty},\operatorname{Hom}(V,W_2))
		\end{tikzcd}
		$$
		which commutes and has exact rows. By the inductive hypothesis the first vertical arrow is an isomorphism, and the second and fourth are injective. A diagram chase shows that the third is injective also. Similarly we obtain the diagram
		$$
		\hspace*{-1cm}
		\begin{tikzcd}[column sep =tiny, row sep=small]
		H^0(G_K,\operatorname{Hom}(V,W_1)) \arrow{r} \arrow{d} & H^0(G_K,\operatorname{Hom}(V,W)) \arrow{r} \arrow{d} & H^0(G_K,\operatorname{Hom}(V,W_2)) \arrow{r} \arrow{d} & H^1(G_K,\operatorname{Hom}(V,W_1)) \arrow{d} \\
		H^0(G_{K_\infty},\operatorname{Hom}(V,W_1) \arrow{r}  & H^0(G_{K_\infty},\operatorname{Hom}(V,W)) \arrow{r}  & H^0(G_{K_\infty},\operatorname{Hom}(V,W_2)) \arrow{r}  & H^1(G_{K_\infty},\operatorname{Hom}(V,W_1))
		\end{tikzcd}
		$$
		which commutes and has exact rows. Again the inductive hypothesis implies the first and third vertical arrows are isomorphisms, and the fourth is injective. A diagram chase shows the second is an isomorphism, which finishes the proof. 
\end{proof}
\begin{corollary}\label{consequences}
	Let $V$ and $W$ be as in Lemma~\ref{devissage} and let $0 \rightarrow W \rightarrow Z \rightarrow V \rightarrow 0$ be an exact sequence in $\operatorname{Rep}(G_K)$. Then any $G_{K_\infty}$-equivariant splitting $s\colon V \rightarrow Z$ of this sequence is $G_K$-equivariant.
\end{corollary}
\begin{proof}
	The $i=1$ part of Lemma~\ref{devissage} implies $0 \rightarrow W \rightarrow Z \rightarrow V \rightarrow 0$ is split in $\operatorname{Rep}(G_K)$, say by a $G_K$-equivariant map $s'\colon V \rightarrow Z$. Then $s' - s \in H^0(G_{K_\infty},\operatorname{Hom}(V,W))$ and so by the $i=0$ part of Lemma~\ref{devissage} implies $s'-s \in H^0(G_K,\operatorname{Hom}(V,W))$. Thus $s$ is $G_K$-equivariant.
\end{proof}
\subsection{Proving the theorem} 
\begin{lemma}\label{inductivestep}
	Let $V,W$ be cyclotomic-free $G_K$-representations, with $V$ irreducible. Then any $G_{K_\infty}$-equivariant map $f \colon V \rightarrow W$ is $G_K$-equivariant.
\end{lemma}
\begin{proof}
	Let $(W_j)_j$ be a composition series as in Definition~\ref{cycfreecondition}. There is a largest $j$ such that $f$ factors through $W_j \hookrightarrow W$; so long as $f \neq 0$ (in which case the lemma is trivial) the composite $g\colon V \xrightarrow{f} W_j \rightarrow W_j/W_{j+1}$ is then non-zero. Lemma~\ref{equiv} implies $g$ is a $G_K$-equivariant isomorphism. Now $f \circ g^{-1}$ is a $G_{K_\infty}$-splitting of $0 \rightarrow W_{j+1} \rightarrow W_j \rightarrow W_j/W_{j+1} \rightarrow 0$ and so $f \circ g^{-1}$ is $G_K$-equivariant by Corollary~\ref{consequences}. Thus $f$ is $G_K$-equivariant.
\end{proof}
\begin{proof}[Proof of Theorem~\ref{ff}]
	Induct on the length of $V$. When $V$ is irreducible the theorem is given by Lemma~\ref{equiv}. If $V$ is not irreducible choose a composition series $(V_i)_i$ as in Definition~\ref{cycfreecondition}. As $V_1$ is cyclotomic-free our inductive hypothesis implies $f\colon V_1 \rightarrow W$ is $G_K$-equivariant. In particular $\operatorname{ker}(f|_{V_1}) \subset V$ is $G_K$-stable. If $\operatorname{ker}(f|_{V_1}) \neq 0$ then, as $f$ factors through $V \rightarrow V/\operatorname{ker}(f|_{V_1})$ and $V/\operatorname{ker}(f|_{V_1})$ is cyclotomic-free by Lemma~\ref{quotient}, the result follows from our inductive hypothesis. Thus we can assume $f|_{V_1}$ is injective. Set $W_1 = f(V_1)$; again since $f|_{V_1}$ is $G_K$-equivariant $W_1 \subset W$ is $G_K$-stable. 
	
	Now consider the commutative diagram
	$$
	\begin{tikzcd}[column sep =tiny, row sep=small]
	0 \arrow{r}  & V_{1} \arrow{r} \arrow{d} & V \arrow{r} \arrow{d}{f} & V/V_{1} \arrow{d} \arrow{r} & 0 \\
	0 \arrow{r} & W_{1} \arrow{r}  & W \arrow{r}  & W/W_{1} \arrow{r}  & 0
	\end{tikzcd}
	$$
	with the rows exact in $\operatorname{Rep}(G_K)$. As $W/W_1$ is cyclotomic-free by Lemma~\ref{quotient}, $V/V_1 \rightarrow W/W_1$ is $G_K$-equivariant. Thus $f(\sigma(v)) - \sigma(f(v)) \in W_1$ for all $v \in V$ and $\sigma \in G_K$, and so $f(V) \subset W$ is $G_K$-stable. As $f(V)$ is cyclotomic-free we may assume $f$ is surjective. This implies $V/V_1 \rightarrow W/W_1$ is surjective, and so is a $G_K$-isomorphism since $V/V_1$ is irreducible.
	
	If we identify $V/V_{n-1} = W/W_{n-1}$ and $V_{n-1} = W_{n-1}$ via the outer arrows of the above diagram, the map $f$ shows that the two horizontal rows in the diagram define the same class in $\operatorname{Ext}^1_{\operatorname{Rep}(G_{K_\infty})}(V/V_{n-1},V_{n-1}) = H^1(G_{K_\infty},\operatorname{Hom}(V/V_{n-1},V_{n-1}))$. Lemma~\ref{devissage} therefore implies these two rows define the same class in $\operatorname{Ext}^1_{\operatorname{Rep}(G_K)}(V/V_{n-1},V_{n-1})$ and so there exists a $G_K$-equivariant $h\colon V \rightarrow W$ fitting into the diagram as $f$ does. Thus $h -f$ describes a $G_{K_\infty}$-equivariant map $V/V_1 \rightarrow W_1$. By induction $h-f$ is $G_K$-equivariant and so $f$ is $G_K$-equivariant also.
\end{proof}

\subsection{Stable lattices} To conclude this section fix a finite extension $E/\mathbb{Q}_p$ with ring of integers $\mathcal{O}$ and residue field $\mathbb{F}$. By using the above we shall show that certain $G_{K_\infty}$-stable lattices inside $E$-representations of $G_K$ are $G_K$-stable.

If $G$ is a topological group acting linearly and continuously on a topological abelian group $M$ then let $Z^1(G,M)$ denote the group of continuous $1$-cocycles $G \rightarrow M$, and let $B^1(G,M) \subset Z^1(G,M)$ denote the group of $1$-coboundaries. 
\begin{lemma}\label{coboundary}
Let $V,W$ be continuous representations of $G_K$ on finite free $\mathcal{O}$-modules. Assume $\overline{W} = W \otimes_{\mathcal{O}} \mathbb{F}$ is cyclotomic-free, that $\overline{V} = V \otimes_{\mathcal{O}} \mathbb{F}$ is irreducible, and that $\overline{V} \otimes_{\mathbb{F}} \mathbb{F}(1)$ is not a Jordan--Holder factor of $\overline{W}$.

If $c \in Z^1(G_K,\operatorname{Hom}(V,W)[\frac{1}{p}])$ is such that $c(\sigma) \in \operatorname{Hom}(V,W)$ for all $\sigma \in G_{K_\infty}$ then $c \in Z^1(G_K,\operatorname{Hom}(V,W))$.
\end{lemma}
\begin{proof}
Put $H = \operatorname{Hom}(V,W)$ (the representation of $\mathcal{O}$-linear homomorphisms) and $\overline{H} = \operatorname{Hom}(\overline{V},\overline{W}) = H \otimes_{\mathcal{O}} \mathbb{F}$. Lemma~\ref{devissage} implies $H^0(G_K,\overline{H}) = H^0(G_{K_\infty},\overline{H})$ and $H^1(G_K,\overline{H}) \rightarrow H^1(G_{K_\infty},\overline{H})$ is injective.

Let $\mathfrak{m}$ denote the maximal ideal of $\mathcal{O}$. We claim there exists $\mu \in \mathfrak{m}$ such that $\mu c \in Z^1(G_K,H)$. Recall from \cite[Proposition 2.3]{Tate76} that the map $H^1(G_K,H) \rightarrow H^1(G_K,H[\frac{1}{p}])$, coming from $H \hookrightarrow H[\frac{1}{p}]$, induces an isomorphism $H^1(G_K,H)[\frac{1}{p}] \cong H^1(G_K,H[\frac{1}{p}])$. Thus there exists $\mu' \in \mathfrak{m}$ such that the class of $\mu'c$ in $H^1(G_K,H[\frac{1}{p}])$ is represented by $\widetilde{c} \in Z^1(G_K,H)$. Thus $\mu' c - \widetilde{c} \in B^1(G_K,H[\frac{1}{p}])$. Clearly our claim holds for cocycles in $B^1(G_K,H[\frac{1}{p}])$, so the claim holds in general. 

Choose $\mu \in \mathfrak{m}$ so that $\mu c \in Z^1(G_K,H)$ and so that $v_p(\mu)$ is minimal amongst all such $\mu$. If $v_p(\mu) = 0$ there is nothing to prove so assume $v_p(\mu) >0$. Since $c(\sigma) \in H$ for $\sigma \in G_{K_\infty}$ the image $\overline{\mu c}$ of $\mu c$ in $Z^1(G_K,\overline{H})$ must vanish on $G_{K_\infty}$. Injectivity of $H^1(G_K,\overline{H}) \rightarrow H^1(G_{K_\infty},\overline{H})$ therefore implies $\overline{\mu c} \in B^1(G_K,\overline{H})$, so $\overline{\mu c} = (\sigma -1)\overline{h}$ for some $\overline{h} \in \overline{H}$. As $\overline{\mu c}|_{G_{K_\infty}} =0$, $\overline{h} \in H^0(G_{K_\infty},\overline{H})$ and so $\overline{h} \in H^0(G_K,\overline{H})$. Thus $\overline{\mu c} =0$ on $G_K$ which contradicts the minimality of $v_p(\mu)$.
\end{proof}

\begin{corollary}\label{stable}
	Let $V$ and $W$ be as in Lemma~\ref{coboundary}, and
	\begin{equation}\label{seq}
	0 \rightarrow W \rightarrow Z \rightarrow V \rightarrow 0
	\end{equation}
	be an exact sequence of $G_{K_\infty}$-representations. If the $G_{K_\infty}$-action on $Z[\frac{1}{p}]$ extends to a continuous $G_K$-action so that \eqref{seq} becomes $G_{K}$-equivariant after inverting $p$, then $Z$ is $G_K$-stable inside $Z[\frac{1}{p}]$.
\end{corollary}
\begin{proof}
	In the usual way the extension \eqref{seq} defines a continuous $1$-cocycle $c: G_{K_\infty} \rightarrow \operatorname{Hom}(V,W)$ in $Z^1(G_{K_\infty},\operatorname{Hom}(V,W))$. The extension of \eqref{seq} to $G_K$ after inverting $p$ produces an extension of $c$ to an element of $Z^1(G_K,\operatorname{Hom}(V,W)[\frac{1}{p}])$. To show $Z$ is $G_K$-stable it suffices to show this extension of $c$ lies in $Z^1(G_K,\operatorname{Hom}(V,W))$, and this follows from Lemma~\ref{coboundary}.
\end{proof}

\section{Breuil--Kisin modules}\label{section3}
In this section we recall the theory of Breuil--Kisin modules and their relationship to crystalline Galois representations.

Throughout let $k$ denote a finite field of characteristic $p$. Write $K_0 = W(k)[\frac{1}{p}]$ and let $K$ be a totally ramified extension of $K_0$ of degree $e$. Let $C$ denote the completion of a fixed algebraic closure of $K$, with ring of integers $\mathcal{O}_C$.
\subsection{Breuil--Kisin modules}
	Let $\mathfrak{S} = W(k)[[u]]$. We equip this ring with the $\mathbb{Z}_p$-linear endomorphism $\varphi$ which acts on $W(k)$ by the Witt vector Frobenius and which sends $u \mapsto u^p$. Fix a uniformiser $\pi \in K$ and let $E(u) \in \mathfrak{S}$ denote the minimal polynomial of $\pi$ over $K_0$.
	\begin{definition}
		A Breuil--Kisin module $M$ is a finitely generated $\mathfrak{S}$-module equipped with an isomorphism
		$$
		\varphi_M: M \otimes_{\varphi,\mathfrak{S}} \mathfrak{S}[\tfrac{1}{E}] \cong M[\tfrac{1}{E}]
		$$
		When there is no risk of confusion we write $\varphi$ in place of $\varphi_M$. We can identify $\varphi_M$ with the semilinear map $M \rightarrow M[\frac{1}{E}]$ given by $m \mapsto \varphi_M(m \otimes 1)$. Denote the category of Breuil--Kisin modules by $\operatorname{Mod}^{\operatorname{BK}}_K$.
	\end{definition}
	
	We now recall the connection between Breuil--Kisin modules and crystalline representations. As in the previous section choose a compatible system $\pi^{1/p^n}$ of $p^n$-th roots of $\pi$ in $C$, and let $K_\infty = K(\pi^{1/p^\infty})$. Let $A_{\operatorname{inf}} =W(\mathcal{O}_{C^\flat})$ where $\mathcal{O}_{C^\flat} = \varprojlim \mathcal{O}_C/p$ with transition maps given by $x \mapsto x^p$. Our choice of $\pi^{1/p^n}$ defines an element $\pi^\flat \in \mathcal{O}_{C^\flat}$ and we embed $\mathfrak{S} \rightarrow A_{\operatorname{inf}}$ via $\sum a_i u^i \mapsto \sum a_i [\pi^\flat]^i$. This inclusion is compatible with $\varphi$ on $\mathfrak{S}$ and the Witt vector Frobenius on $A_{\operatorname{inf}}$. Note that the $G_K$-action on $\mathcal{O}_C/p$ induces a $G_K$-action on $A_{\operatorname{inf}}$, and via this action the image of $\mathfrak{S} \rightarrow A_{\operatorname{inf}}$ is $G_{K_\infty}$-stable

	\begin{lemma}\label{BKF}
		There is an exact functor $M \mapsto T(M) =  (M \otimes_{\mathfrak{S}} W(C^\flat))^{\varphi =1}$ from $\operatorname{Mod}^{\operatorname{BK}}_K$ to the category of finitely generated $\mathbb{Z}_p$-modules equipped with a continuous $\mathbb{Z}_p$-linear action of $G_{K_\infty}$. Further, there are $\varphi,G_{K_\infty}$-equivariant identifications
		$$
		M \otimes_{\mathfrak{S}} W(C^\flat) \cong T(M) \otimes_{\mathbb{Z}_p} W(C^\flat)
		$$
		which are functorial in $M$.
	\end{lemma}
	\begin{proof}
		This is proven in \cite{Fon00}. We refer to \cite[Proposition 4.1.5]{B18} and \cite[Construction 4.2.3]{B18} for more details.
	\end{proof}
	
	In the following theorem a crystalline $\mathbb{Z}_p$-lattice is a $G_K$-stable $\mathbb{Z}_p$-lattice inside a crystalline (in the sense of \cite{Fon94b}) $\mathbb{Q}_p$-representation of $G_K$. 

\begin{theorem}[Kisin]\label{kisin}
	There is a fully faithful functor $T \mapsto M(T)$ from the category of crystalline $\mathbb{Z}_p$-lattices into the category of Breuil--Kisin modules finite free over $\mathfrak{S}$. The module $M(T)$ is characterised up to isomorphism by $T(M(T)) \cong T$ as $G_{K_\infty}$-representations. 
\end{theorem}
\begin{proof}
	The functor $T \mapsto M(T)$ was first constructed by Kisin in \cite{Kis06}. The formulation we have given here is that in \cite[Theorem 4.4]{BMS}.
\end{proof}
\subsection{Coefficients} In practice it is sometimes necessary to consider crystalline representations valued in extensions of $\mathbb{Z}_p$. Kisin's construction can be suitably adapted to allow this, provided the coefficient ring is finite over $\mathbb{Z}_p$.

\begin{notation}
	Let $E/\mathbb{Q}_p$ be a finite extension with ring of integers $\mathcal{O}$ and residue field $\mathbb{F}$. Assume that $K_0 \subset E$. 
\end{notation}

By a crystalline $\mathcal{O}$-lattice we mean an $\mathcal{O}$-lattice inside a $G_K$-representation on an $E$-vector space, which is crystalline when viewed as a representation on a $\mathbb{Q}_p$-vector space. By functoriality, if $T$ is a crystalline $\mathcal{O}$-lattice then $M(T)$ is a Breuil--Kisin module with $\mathcal{O}$-action as defined below.
\begin{definition}
	A Breuil--Kisin module with $\mathcal{O}$-action is a pair $(M,\iota)$ where $M \in \operatorname{Mod}^{\operatorname{BK}}_K$ and $\iota$ is a $\mathbb{Z}_p$-algebra homomorphism $\iota \colon\mathcal{O} \rightarrow \operatorname{End}_{\operatorname{BK}}(M)$. Equivalently a Breuil--Kisin module with $\mathcal{O}$-action is a finitely generated $\mathfrak{S}_{\mathcal{O}}:=\mathfrak{S} \otimes_{\mathbb{Z}_p} \mathcal{O}$-module equipped with an isomorphism
	$$
	M \otimes_{\varphi \otimes 1, \mathfrak{S}_{\mathcal{O}}} \mathfrak{S}_{\mathcal{O}}[\tfrac{1}{E}] \cong M[\tfrac{1}{E}]
	$$
	Let $\operatorname{Mod}^{\operatorname{BK}}_K(\mathcal{O})$ denote the category of Breuil--Kisin modules with $\mathcal{O}$-action.
\end{definition}

\begin{construction}\label{O-semilinear}
	Our assumption that $K_0 \subset E$ has the following consequence. Since $\mathfrak{S} \otimes_{\mathbb{Z}_p} \mathcal{O}$ is a finite $\mathfrak{S}$-module it is $u$-adically complete, and so the inclusion $\mathcal{O}[u] \rightarrow \mathfrak{S} \otimes_{\mathbb{Z}_p} \mathcal{O}$ given by $\sum a_iu^i \mapsto \sum u^i \otimes a_i$ extends to $\mathcal{O}[[u]] \rightarrow \mathfrak{S} \otimes_{\mathbb{Z}_p} \mathcal{O}$. In this way we view $\mathfrak{S} \otimes_{\mathbb{Z}_p} \mathcal{O}$ as an $\mathcal{O}[[u]]$-module. The map
	$$
	(\sum a_i u^i) \otimes b \mapsto (\sum \tau(a_i)bu^i)_\tau
	$$
	then describes an isomorphism of $\mathcal{O}[[u]]$-algebras $\mathfrak{S} \otimes_{\mathbb{Z}_p} \mathcal{O} \rightarrow \prod_{\tau} \mathcal{O}[[u]]$, the product running over $\tau \in \operatorname{Hom}_{\mathbb{F}_p}(k,\mathbb{F})$ (we abusively write $\tau$ also for its extension to an embedding $\tau\colon W(k) \rightarrow \mathcal{O}$). Let $\widetilde{e}_\tau \in \mathfrak{S} \otimes_{\mathbb{Z}_p} \mathcal{O}$ be the idempotent corresponding to $\tau$. As $\widetilde{e}_\tau$ is determined by the property $(a \otimes 1)\widetilde{e}_\tau = (1 \otimes \tau(a))\widetilde{e}_\tau$ for $a \in W(k)$, the map $\varphi \otimes 1$ sends
	$$
	\widetilde{e}_{\tau \circ \varphi} \mapsto \widetilde{e}_\tau
	$$
	If $M \in \operatorname{Mod}^{\operatorname{BK}}_K(\mathcal{O})$ we set $M_\tau = \widetilde{e}_\tau M$ which we view as an $\mathcal{O}[[u]]$-algebra. By the above $\varphi_M$ restricts to a map
	\begin{equation}\label{Olalal}
	M_{\tau \circ \varphi} \otimes_{\varphi,\mathcal{O}[[u]]} \mathcal{O}[[u]] \rightarrow M_\tau[\tfrac{1}{\tau(E)}]
	\end{equation}
	which becomes an isomorphism after inverting $\tau(E)$. Here $\varphi$ on $\mathcal{O}[[u]]$ is that induced by $\varphi \otimes 1$ on $\mathfrak{S} \otimes_{\mathbb{Z}_p} \mathcal{O}$, i.e. is given by $\sum a_i u^i \mapsto \sum a_iu^{ip}$.
\end{construction}

\begin{corollary}\label{freecoeff}
	If $M \in \operatorname{Mod}^{\operatorname{BK}}_K(\mathcal{O})$ is free as an $\mathfrak{S}$-module then $M$ is free over $\mathfrak{S} \otimes_{\mathbb{Z}_p} \mathcal{O}$.
\end{corollary}
\begin{proof}
	If $M$ is free over $\mathfrak{S}$ then the $\mathcal{O}$-module $M/uM$ is free over $W(k)$, is torsion-free, and is therefore free over $\mathcal{O}$. By Nakayama's lemma, any lift of an $\mathcal{O}$-basis of $M/uM$ generates $M$; there is therefore a surjection $F \rightarrow M$ where $F$ is $\mathcal{O}[[u]]$-free of $\mathfrak{S}$-rank equal to that of $M$. As surjective maps between free-modules of the same rank are isomorphisms, $M$ is free over $\mathcal{O}[[u]]$. By \eqref{Olalal} each $M_\tau$ has the same $\mathcal{O}[[u]]$-rank which proves $M$ is free over $\mathfrak{S} \otimes_{\mathbb{Z}_p} \mathcal{O} = \prod_{\tau} \mathcal{O}[[u]]$.
\end{proof}
\subsection{Strong divisibility} We now explain what it means for a $p$-torsion Breuil--Kisin module to be strongly divisible. After a result of Gee--Liu--Savitt (Theorem~\ref{GLS}) strong divisibility is closely related to the reduction modulo~$p$ of crystalline representations.

Note that since $E \equiv u^e$ modulo~$p$, a $p$-torsion Breuil--Kisin module is a finitely generated $k[[u]]$-module $M$ equipped with an isomorphism $M \otimes_{\varphi,k[[u]]} k((u)) \rightarrow M[\frac{1}{u}]$. 

From now on all Breuil--Kisin modules will be considered with $\mathcal{O}$-action.

\begin{definition}
Let $\operatorname{Mod}^{\operatorname{BK}}_k(\mathcal{O})$ be the full sub-category of $M \in \operatorname{Mod}^{\operatorname{BK}}_K(\mathcal{O})$ which are free modules over $k[[u]] \otimes_{\mathbb{F}_p} \mathbb{F}$.
\end{definition}

If $M \in \operatorname{Mod}^{\operatorname{BK}}_k(\mathcal{O})$ set $M^\varphi$ equal to the image of $ M \otimes_{\varphi,k[[u]]} k[[u]] \xrightarrow{\varphi} M[\frac{1}{u}]$. Equivalently $M^\varphi$ is the $k[[u]]$-sub-module of $M[\frac{1}{u}]$ generated by the $k[[u^p]]$-module $\varphi(M)$

\begin{construction}
	Equip $M^\varphi$ with the filtration $F^i M^\varphi = M^\varphi \cap u^iM$. Similarly define a filtration on $M$ by $F^iM = \lbrace m \in M \mid \varphi(m) \in u^iM \rbrace$. Note that the semi-linear map $\varphi\colon M \rightarrow M^\varphi$ is compatible with these filtrations.
	
	Set $M^\varphi_k = M^\varphi/uM^\varphi$ and $M_k = M/uM$. These are both $k \otimes_{\mathbb{F}_p} \mathbb{F}$-modules and we equip both with the quotient filtration coming from $M^\varphi$ and $M$ respectively. In other words $F^i M^\varphi_k$ equals the image of $F^iM^\varphi$ under $M^\varphi \rightarrow M^\varphi_k$, and likewise $F^i M_k$ is the image of $F^iM$ under $M \rightarrow M_k$.
\end{construction}

The filtration on $M_k$ is by $k \otimes_{\mathbb{F}_p} \mathbb{F}$-sub-modules. Thus, as in Construction~\ref{O-semilinear}, there are decompositions $M_k = \prod_{\tau \in \operatorname{Hom}_{\mathbb{F}_p}(k,\mathbb{F})} M_{k,\tau}$ of filtered modules. Likewise $M_k^\varphi = \prod_{\tau} M_{k,\tau}^\varphi$. Each of $M_{k,\tau}$ and $M_{k,\tau}^\varphi$ is a filtered $\mathbb{F}$-vector space. 
\begin{remark}
	Note that the map $\varphi \colon M \rightarrow M^\varphi$ induces a semi-linear map $M_k \rightarrow M^\varphi_k$, which is compatible with filtrations. This latter map is a bijection but not necessarily an isomorphism of filtered modules. 
\end{remark}
\begin{lemma}\label{equivcond}
	The map $M_k \rightarrow M^\varphi_k$ is a semi-linear isomorphism of filtered modules if and only if for each $\tau \in \operatorname{Hom}_{\mathbb{F}_p}(k,\mathbb{F})$ there exist an $\mathbb{F}[[u]]$-basis $(f_i)_i$ of $M_\tau$ and integers $(r_i)$ such that $(u^{r_i}f_i)$ is an $\mathbb{F}[[u^p]]$-basis of $\varphi(M)_\tau = \varphi(M_{\tau \circ \varphi})$. 
\end{lemma}
\begin{proof}
	This is \cite[Lemma 5.3.4]{B18}.
\end{proof}
\begin{definition}
	For $M \in \operatorname{Mod}^{\operatorname{BK}}_k(\mathcal{O})$ and $\tau \in \operatorname{Hom}_{\mathbb{F}_p}(k,\mathbb{F})$ define $\operatorname{Weight}_\tau(M)$ to be the multiset of integers which contains $i$ with multiplicity 
	$$
	\operatorname{dim}_{\mathbb{F}} \operatorname{gr}^i(M^\varphi_k)
	$$
\end{definition} 
\begin{definition}[Strong divisibility]
	$M \in \operatorname{Mod}^{\operatorname{BK}}_k(\mathcal{O})$ is strongly divisible if $\operatorname{Weight}_\tau(M) \subset [0,p]$ for each $\tau$ and if $M_{k} \rightarrow M^\varphi_{k}$ is a semi-linear isomorphism of filtered modules. Let $\operatorname{Mod}^{\operatorname{SD}}_k(\mathcal{O})$ denote full subcategory of strongly divisible Breuil--Kisin modules.
\end{definition}

\begin{remark}\label{explicit}
	Recall that any matrix $X \in \operatorname{GL}_n(\mathbb{F}((u)))$ can be written uniquely as $A \operatorname{diag}(u^{r_i})B$ for some $r_i \in \mathbb{Z}$ and $A,B \in \operatorname{GL}_n(\mathbb{F}[[u]])$. If $M \in \operatorname{Mod}^{\operatorname{BK}}_k(\mathcal{O})$ and if we choose $\mathbb{F}[[u]]$-bases of $M_{\tau \circ \varphi}$ and $M_\tau$ then the with respect to these bases $\varphi: M_{\tau \circ \varphi} \rightarrow M_\tau$ may be represented by a matrix 
	$$
	A \operatorname{diag}(u^{r_i})B
	$$
	with $A,B \in \operatorname{GL}_n(\mathbb{F}[[u]])$. Then $\lbrace r_i \rbrace = \operatorname{Weight}_\tau(M)$. Moreover the conditions of Lemma~\ref{equivcond} are equivalent to asking that $B \in \operatorname{GL}_n(\mathbb{F}[[u^p]])$. In particular $M \in \operatorname{Mod}^{\operatorname{SD}}_k(\mathcal{O})$ if and only if $r_i \in [0,p]$ and $B \in \operatorname{GL}_n(\mathbb{F}[[u^p]])$.
\end{remark}

\begin{proposition}\label{SDexact}
	Let $0 \rightarrow M \rightarrow N \rightarrow P \rightarrow 0$ be an exact sequence in $\operatorname{Mod}^{\operatorname{BK}}_k(\mathcal{O})$.
	\begin{enumerate}
		\item If $N \in \operatorname{Mod}^{\operatorname{SD}}_k(\mathcal{O})$ then $M,P \in \operatorname{Mod}^{\operatorname{SD}}_k(\mathcal{O})$ and $\operatorname{Weight}_\tau(N) = \operatorname{Weight}_\tau(M) \cup \operatorname{Weight}_\tau(P)$.
		\item If $M,P \in \operatorname{Mod}^{\operatorname{SD}}_k(\mathcal{O})$ then $N \in \operatorname{Mod}^{\operatorname{SD}}_k(\mathcal{O})$ if and only if the map $N \rightarrow P$ is a strict\footnote{Recall that a map $f\colon M \rightarrow N$ of filtered modules is strict if $f(F^iM) = F^i N \cap f(M)$ for every $i \in \mathbb{Z}$.} map of filtered modules.
	\end{enumerate}
\end{proposition}
\begin{proof}
	This is \cite[Proposition 5.4.7]{B18}.
\end{proof}

Recall that if $V$ is a crystalline representation of $G_K$ on an $E$-vector space and $D_{\operatorname{crys}}(V) = (V \otimes_{\mathbb{Q}_p} B_{\operatorname{crys}})^{G_K}$ is the associated filtered $\varphi$-module, then $D_{\operatorname{crys}}(V)$ is a free $K_0 \otimes_{\mathbb{Q}_p} E$-module, and so $D_{\operatorname{crys}}(V)_K = \prod_\tau D_{\operatorname{crys}}(V)_{K,\tau}$ as filtered modules. Each $D_{\operatorname{crys}}(V)_{K,\tau}$ is a $K \otimes_{K_0} E$-module; the $\tau$-th Hodge--Tate weights of $V$ is the multiset $\operatorname{HT}_\tau(V)$ containing $i$ with multiplicity 
$$
\operatorname{dim}_E \operatorname{gr}^i(D_{\operatorname{crys}}(V)_{K,\tau})
$$ 
Thus $\operatorname{HT}_\tau(V)$ contains $e\operatorname{dim}_E V$ integers and our normalisations are such that the Hodge--Tate weight of the cyclotomic character is $-1$ (or rather $e$ copies of $-1$).

The following theorem relates $\operatorname{Mod}^{\operatorname{SD}}_k(\mathcal{O})$ to reductions of crystalline representations. We must assume that $K = K_0$ and that if $p =2$ then $\pi$ is chosen so that $K_\infty \cap K(\mu_{p^\infty}) = K$ (such $\pi$ exist by \cite[Lemma 2.1]{Wang17}).
\begin{theorem}[Gee--Liu--Savitt, Wang]\label{GLS}
	Assume $K$ and $\pi$ are as in the previous paragraph. Let $T$ be a crystalline $\mathcal{O}$-lattice and let $V = T \otimes_{\mathcal{O}} E$. If $\operatorname{HT}_\tau(V) \subset [0,p]$ for each $\tau$ then $\overline{M} := M(T) \otimes_{\mathcal{O}} \mathbb{F} \in \operatorname{Mod}^{\operatorname{SD}}_k(\mathcal{O})$ and $\operatorname{Weight}_\tau(\overline{M}) = \operatorname{HT}_\tau(V)$.
\end{theorem}
\begin{proof}
	When $p>2$ this follows by reducing the description of $M(T)$ given in \cite[Theorem 4.22]{GLS} modulo a uniformizer of $\mathcal{O}$. The case $p=2$ is proven in \cite[Theorem 4.2]{Wang17}.\footnote{When using results from \cite{GLS,Wang17} it is important to keep track of normalisations. In both references Hodge--Tate weights are normalised to be the opposite of ours. Also Breuil--Kisin modules are attached contravariantly to crystalline representations; from the Breuil--Kisin module $\mathfrak{M}$ associated to $T$ in \cite{GLS,Wang17} one recovers $M(T)$ as the dual of $\mathfrak{M}$ (for the dual of a Breuil--Kisin module see the construction at the start of Subsection~\ref{cohom})}
\end{proof}

\section{Strongly divisible extensions}\label{section4}
We maintain the notation from the previous subsection. Our aim here is to compute dimensions of the space of extensions of strongly divisible Breuil--Kisin modules.

Throughout we shall use the following construction. If $M,N \in \operatorname{Mod}^{\operatorname{BK}}_K(\mathcal{O})$ define a Breuil--Kisin module $\operatorname{Hom}(M,N)^{\mathcal{O}}$ with underlying module $\operatorname{Hom}_{\mathfrak{S}_{\mathcal{O}}}(M,N)$ and with Frobenius given by $f \mapsto \varphi_N \circ f \circ \varphi_M^{-1}$ (see \cite[Construction 4.2.5 and 4.3.3]{B18}).

\subsection{Cohomology and ext groups}\label{cohom}
	If $M \in \operatorname{Mod}^{\operatorname{BK}}_k(\mathcal{O})$ let $H^i(M)$ denote the cohomology of the complex
	$$
	M \xrightarrow{\varphi - 1} M[\tfrac{1}{u}]
	$$
	The $H^i(M)$ are $\mathbb{F}$-vector spaces. 

	\begin{construction}\label{H1andExt}
		If $P,M \in \operatorname{Mod}^{\operatorname{BK}}_k(\mathcal{O})$ then there is a map
		\begin{equation}\label{identO}
		H^1(\operatorname{Hom}(P,M)^{\mathcal{O}}) \rightarrow \operatorname{Ext}^1_{\mathbb{F}}(P,M)
		\end{equation}
		into the first Yoneda extension group in the exact category $\operatorname{Mod}^{\operatorname{BK}}_k(\mathcal{O})$. This map sends a class represented by $f \in \operatorname{Hom}(P,M)^{\mathcal{O}}[\frac{1}{u}]$ onto a class represented by an extension $0 \rightarrow M \rightarrow N_{f} \rightarrow P \rightarrow 0$ in $\operatorname{Mod}^{\operatorname{BK}}_k(\mathcal{O})$ where $N_{f}$ is the Breuil--Kisin module with underlying $k[[u]] \otimes_{\mathbb{F}_p} \mathbb{F}$-module $M \oplus P$ with Frobenius given by $(\varphi_M + f \circ \varphi_P, \varphi_P)$. This map is injective and functorial in $P$ and $M$, in particular it is a map of $\mathbb{F}$-vector spaces. Since every extension in $\operatorname{Mod}^{\operatorname{BK}}_k(\mathcal{O})$ of $P$ by $M$ splits as a $k[[u]] \otimes_{\mathbb{F}_p} \mathbb{F}$-module \eqref{identO} is surjective, and so an isomorphism.
\end{construction}

	If $M \in \operatorname{Mod}^{\operatorname{BK}}_k(\mathcal{O})$ let $H^i_{\operatorname{SD}}(M)$ denote the cohomology of the complex
	$$
	F^0M \xrightarrow{\varphi -1} M
	$$
	Then $H^0_{\operatorname{SD}}(M) = H^0(M)$. The inclusion $M \rightarrow M[\frac{1}{u}]$ induces a map $H^1_{\operatorname{SD}}(M) \rightarrow H^1(M)$. If $m \in M$ can be written as $\varphi(m')-m'$ with $m' \in M$ then $\varphi(m') \in M$ and so $m' \in F^0M$; therefore $H^1_{\operatorname{SD}}(M) \rightarrow H^1(M)$ is injective. Let $$
	\operatorname{Ext}^1_{\operatorname{SD}}(P,M) \subset \operatorname{Ext}^1_{\mathbb{F}}(P,M)
	$$
	denote the image of $H^1_{\operatorname{SD}}(\operatorname{Hom(P,M)^{\mathcal{O}}})$ under \eqref{identO}.

\begin{lemma}\label{SD=SD}
	Let $0 \rightarrow M \rightarrow N \rightarrow P \rightarrow 0$ be an extension in $\operatorname{Mod}^{\operatorname{BK}}_k(\mathcal{O})$ and suppose that $M,P \in \operatorname{Mod}^{\operatorname{SD}}_k(\mathcal{O})$. Then $N \in \operatorname{Mod}^{\operatorname{SD}}_k(\mathcal{O})$ if and only if the class of this extension lies in $\operatorname{Ext}^1_{\operatorname{SD}}(P,M)$.
\end{lemma}  
\begin{proof}
	As in Construction~\ref{H1andExt} we can assume $N=N_{f}$ for some $f \in \operatorname{Hom}(P,M)^{\mathcal{O}}[\frac{1}{u}]$. Thus as a module $N = M \bigoplus P$ and $\varphi_N = (\varphi_M + f \circ \varphi_P,\varphi_P)$. Proposition~\ref{SDexact} implies $N \in \operatorname{Mod}^{\operatorname{SD}}_k(\mathcal{O})$ if and only if $N \rightarrow P$ is strict as a map of filtered modules. We have to show that if $f \in \operatorname{Hom}(P,M)^{\mathcal{O}}$ then $N \rightarrow P$ is strict, and conversely that $N \rightarrow P$ being strict implies there exists $g \in \operatorname{Hom}(P,M)^{\mathcal{O}}$ such that $f + \varphi(g) -g \in \operatorname{Hom}(P,M)^{\mathcal{O}}$.
	
	The map $N \rightarrow P$ is strict if and only if for every $\tau \in \operatorname{Hom}_{\mathbb{F}_p}(k,\mathbb{F})$ and every $z \in F^iP_{\tau \circ \varphi}$, there exists $(m,z) \in N_{\tau \circ \varphi}$ such that 
	$$
	\varphi((m,z)) = (\varphi_M(m)+f(\varphi_P(z)),\varphi_P(z)) \in u^iN_\tau
	$$
	If $f \in \operatorname{Hom}(P,M)^{\mathcal{O}}$ then $f(\varphi_P(z)) \in u^iM_\tau$, since $\varphi_P(z) \in u^iP_\tau$, and so we can take $m = 0$. This shows $f \in \operatorname{Hom}(P,M)^{\mathcal{O}}$ implies $N \rightarrow P$ is strict.
	
	For the converse, since $P \in \operatorname{Mod}^{\operatorname{SD}}_k(\mathcal{O})$ we can find, for each $\tau \in \operatorname{Hom}_{\mathbb{F}_p}(k,\mathbb{F})$, a basis $(z_i)$ of $P_\tau$ and integers $r_i$ such that $u^{r_i}z_i$ forms a basis of $\varphi_P(P)_\tau$ (Lemma~\ref{equivcond}). If $N \rightarrow P$ is strict then, as in the previous paragraph, we may choose $m_i \in M_{\tau \circ \varphi}$ such that $\varphi_M(m_i) + f(u^{r_i}z_i) \in u^{r_i}M_\tau$. Since the $u^{r_i} z_i$ form an $\mathbb{F}[[u^p]]$-basis of $\varphi_P(P_{\tau \circ \varphi})$, the $n_i = \varphi_P^{-1}(u^{r_i}z_i)$ form an $\mathbb{F}[[u]]$-basis of $P_{\tau \circ \varphi}$. Define $g \in \operatorname{Hom}(P,M)^{\mathcal{O}}$ by asserting that on $P_{\tau \circ \varphi}$ this map sends $n_i \mapsto m_i$. Then $ f +\varphi(g)- g$ sends
	$$
	u^{r_i}z_i \mapsto f(u^{r_i}z_i) + \varphi_M \circ g \circ \varphi^{-1}_P(u^{r_i}z_i) - g(z_i) = f(u^{r_i}z_i) + \varphi_M(m_i) + g(u^{r_i}z_i) \in u^{r_i}M_\tau
	$$
	Thus $f + \varphi(g) - g \in \operatorname{Hom}(P,M)^{\mathcal{O}}$.
\end{proof}

\subsection{Dimension calculations}
We now compute the dimensions of $H^1_{\operatorname{SD}}$. Our proof will use that for $M \in \operatorname{Mod}^{\operatorname{BK}}_k(\mathcal{O})$ there are exact sequences
$$
0 \rightarrow \operatorname{gr}^{i-p}(M) \xrightarrow{u} \operatorname{gr}^i(M) \rightarrow \operatorname{gr}^i(M_k) \rightarrow 0
$$
(here $\operatorname{gr}^i(N) = F^iN/F^{i+1}N$ for any filtered module $N$). The exactness of this sequence follows from the observation $F^iM \cap uM = u(F^{i-p}M)$.

\begin{lemma}\label{dimfirststep}
	Let $M$ be an object of $\operatorname{Mod}^{\operatorname{BK}}_k(\mathcal{O})$. Then both $H^1_{\operatorname{SD}}(M)$ and $H^0(M)$ are finite and
	$$
	\chi(M) - \chi(uM) = \sum_{i \not\in p\mathbb{Z}_{\leq 0} \cup \mathbb{Z}_{\geq 0} } \operatorname{dim}_{\mathbb{F}} \operatorname{gr}^i(M_k)
	$$
	where $\chi(M) := \operatorname{dim}_{\mathbb{F}} H^1_{\operatorname{SD}}(M) - \operatorname{dim}_{\mathbb{F}} H^0(M)$.
\end{lemma}
\begin{proof}
	As $H^0(M) \subset T(M)$ finiteness of $H^0(M)$ is clear. For the rest of the proof consider the inclusion $uM \rightarrow M$. It induces a commutative diagram whose rows are exact.
	$$
	\begin{tikzcd}[column sep =small, row sep=small]
	0 \arrow{r} & F^0(uM) \arrow{r} \arrow{d}{\varphi-1} & F^0M \arrow{r} \arrow{d}{\varphi-1} & Q_1 \arrow{r} \arrow{d}{\alpha} & 0 \\
	0 \arrow{r} & uM \arrow{r} & M \arrow{r} & M_k \arrow{r} & 0
	\end{tikzcd}
	$$
	The snake lemma yields a long exact sequence
	$$
	0 \rightarrow H^0(uM) \rightarrow H^0(M) \rightarrow \operatorname{ker}\alpha \rightarrow H^1_{\operatorname{SD}}(uM) \rightarrow H^1_{\operatorname{SD}}(M) \rightarrow \operatorname{coker}\alpha \rightarrow 0
	$$
	Provided we have finiteness of the $H^1_{\operatorname{SD}}(uM)$ and $H^1_{\operatorname{SD}}(M)$, consideration of the alternating sums of the dimensions in this long exact sequence gives that $\chi(N) - \chi(uN) = \operatorname{dim}_{\mathbb{F}} \operatorname{coker}\alpha - \operatorname{dim}_{\mathbb{F}} \operatorname{ker}\alpha$, which is equal to the $\mathbb{F}$-dimension of $M_k$ minus the $\mathbb{F}$-dimension of $Q_1$. We claim that non-canonically
	\begin{equation}\label{exp}
	Q_1 = \bigoplus_{i \in p\mathbb{Z}_{\leq 0} \cup \mathbb{Z}_{\geq 1}} \operatorname{gr}^i(M_{k})
	\end{equation}
	as an $\mathbb{F}$-vector space. This will imply the second part of the lemma.
	
	To verify \eqref{exp} choose a splitting (as $\mathbb{F}$-vector spaces) of the exact sequence $0 \rightarrow F^1 M \rightarrow F^0M \rightarrow \operatorname{gr}^0(M) \rightarrow 0$. Then we can write $F^0 M  = F^1 M \bigoplus \operatorname{gr}^0(M)$. Observe that $F^0(uM) = (uM) \cap F^1 M$ and that this is the kernel of $F^1 M \rightarrow F^1M_{k}$. Therefore
	$$
	F^0M / F^0(uM) = F^1 M_{k} \bigoplus \operatorname{gr}^0(M)
	$$ 
	Choosing splitting's of $0 \rightarrow F^{i+1} M_{k} \rightarrow F^i M_{k} \rightarrow \operatorname{gr}^i(M_{k}) \rightarrow 0$ allows us to identify  the first term of the above sum with $\bigoplus_{i \in \mathbb{Z}_{\geq 1}} \operatorname{gr}^i(M_{k})$. For the second term: splitting the exact sequence $0 \rightarrow \operatorname{gr}^{i-p}(M) \rightarrow \operatorname{gr}^i(M) \rightarrow \operatorname{gr}^i(M_{k}) \rightarrow 0$ described at the beginning of the subsection shows that $\operatorname{gr}^0(M) = \bigoplus_{i \in p\mathbb{Z}_{\leq 0}} \operatorname{gr}^i(M_k)$. This verifies \eqref{exp}.
	
	It remains to prove that $H^1_{\operatorname{SD}}(M)$ is finite. Observe that, except for $H^1_{\operatorname{SD}}(M)$ and $H^1_{\operatorname{SD}}(uM)$, all the terms in the long exact sequence above are finite. Therefore finiteness of $H^1_{\operatorname{SD}}(M)$ can be deduced from finiteness of $H^1_{\operatorname{SD}}(u^nM)$ for large enough $n$. In fact $H^1_{\operatorname{SD}}(u^nM)$ will vanish for $n$ large enough, as we now show. We need the following lemma, whose proof is straightforward.
	\begin{lemma}\label{multbyu}
		Multiplication by $u$ describes a bijection $F^{i+1-p}(M) \rightarrow F^i(uM)$.
	\end{lemma}
	Continuing with the proof of Lemma~\ref{dimfirststep}, Lemma~\ref{multbyu} implies that $F^1N = N$ when $N = u^nM$ and $n$ is large enough. In this case to show $H^1_{\operatorname{SD}}(N) = 0$ it suffices to show $\varphi -1$ is surjective as a map $N \rightarrow N$. To do this note that for any $x \in N$ we have $\varphi(x) \in uN$ because $F^1N = N$; thus $\sum_{i \geq 0} \varphi^i(-x)$ converges to an element $y \in N$ which satisfies $\varphi(y) - y = x$. This shows surjectivity of $\varphi-1$ and completes the proof.
\end{proof}
\begin{corollary}\label{dimsecondstep}
	Let $M$ be an object of $\operatorname{Mod}^{\operatorname{BK}}_k(\mathcal{O})$ and assume that $\operatorname{gr}^i(M_k) =0$ for $i < -p$. Then
	$$
	\chi(M) = \sum_{i <0} \operatorname{dim}_{\mathbb{F}} \operatorname{gr}^i(M_k)
	$$
	(with $\chi$ as in Lemma~\ref{dimfirststep}).
\end{corollary}
\begin{proof}
	Lemma~\ref{multbyu} implies $\operatorname{dim}_{\mathbb{F}} \operatorname{gr}^i((uM)_k) = \operatorname{dim}_{\mathbb{F}} \operatorname{gr}^{i+1-p}(M_k)$. Note also that $H^0(u^nM) = 0$ for large enough $n$. Thus Lemma~\ref{dimfirststep} shows (without using that $\operatorname{gr}^i(M_k) = 0$ for $i < -p$)
	$$
	\chi(M) = \sum_{n \geq 0} \big( \sum_{i \not\in p\mathbb{Z}_{\leq 0} \cup \mathbb{Z}_{\geq 0}} \operatorname{dim}_{\mathbb{F}} \operatorname{gr}^{i+n(1-p)}(M_k) \big)
	$$
	Since $\operatorname{gr}^i(M_k) = 0$ for $i < -p$ the inner sum for $n=0$ counts the dimensions of $\operatorname{gr}^i(M_k)$ for $i <0$ and $\neq -p$ and the inner sum for $n=1$ counts the dimension of $\operatorname{gr}^{-p}(M_k)$. The remaining inner sums are all zero, which finishes the proof.
\end{proof}
\begin{proposition}\label{DIMFORM}
	If $P$ and $M$ are objects of $\operatorname{Mod}^{\operatorname{SD}}_k(\mathcal{O})$ then
	$$
	\begin{aligned}
	\operatorname{dim}_{\mathbb{F}} \operatorname{Ext}^1_{\operatorname{SD}}(P,M) -& \operatorname{dim}_{\mathbb{F}}\operatorname{Hom}_{\operatorname{BK}}(P,M) =  \\
	& \sum_\tau \operatorname{Card}(\lbrace i-j < 0 \mid i \in \operatorname{Weight}_\tau(M),j \in \operatorname{Weight}_\tau(P) \rbrace)
	\end{aligned}
	$$
\end{proposition}
\begin{proof}
	First we show
	$$
	\lbrace i - j\mid i \in \operatorname{Weight}_{\tau}(M),j \in \operatorname{Weight}_{\tau}(P) \rbrace = \operatorname{Weight}_\tau(\operatorname{Hom}(P,M)^{\mathcal{O}})
	$$
	To see this choose a basis $(m_i)$ of $M_\tau$ such that $(u^{r_i}m_i)$ is a basis of $\varphi(M)_\tau$. The integers $r_i$ are the elements of $\operatorname{Weight}_\tau(M)$. Likewise choose a basis $(p_j)$ of $P_\tau$ such that $(u^{s_j}p_j)$ are a basis of $\varphi(P)_\tau$. One checks that if $f_{ij}$ is the element of $\operatorname{Hom}(P,M)^{\mathcal{O}}$ which is zero everywhere except that it maps $p_j \mapsto m_i$ then the $f_{ij}$ form a basis of $\operatorname{Hom}(P,M)^{\mathcal{O}}_\tau$ and $u^{r_i-s_j}f_{ij}$ forms a basis of $\varphi(\operatorname{Hom}(P,M)^{\mathcal{O}})_\tau$. Now appeal to the comment made after Lemma~\ref{equiv}. 
	
	The previous paragraph shows that $\operatorname{Hom}(P,M)^{\mathcal{O}}$ satisfies the equivalent conditions of Lemma~\ref{equiv} and so $\operatorname{dim}_{\mathbb{F}}\operatorname{gr}^i(\operatorname{Hom}(P,M)^{\mathcal{O}}_k) = \operatorname{dim}_{\mathbb{F}} \operatorname{gr}^i(\operatorname{Hom}(P,M)^{\mathcal{O},\varphi}_k)$. Since $\operatorname{Weight}(\operatorname{Hom}(P,M)^{\mathcal{O}}) \subset [-p,p]$ it follows that $\operatorname{gr}^i(\operatorname{Hom}(P,M)^{\mathcal{O}}_k) = 0$ for $i<-p$. Thus Corollary~\ref{dimsecondstep} applies with $M = \operatorname{Hom}(P,M)^{\mathcal{O}}$. Using Construction~\ref{H1andExt} to identify $\operatorname{Ext}^1_{\operatorname{SD}}(P,M)$ and $H^1_{\operatorname{SD}}(\operatorname{Hom}(P,M)^{\mathcal{O}})$ the result follows.
\end{proof}
\begin{remark}
	This proposition should be compared with the number of possible extensions described in \cite[Theorem 7.9]{GLS}.
\end{remark}

\section{Lifting extensions}\label{section5}
In this section we show how to produce crystalline lifts of some exact sequences in $\operatorname{Mod}^{\operatorname{SD}}_k(\mathcal{O})$. We maintain the notation from the previous two sections.

\subsection{Isogeny categories of Breuil--Kisin modules} One difficulty which arises when producing lifts of extensions in $\operatorname{Mod}^{\operatorname{SD}}_k(\mathcal{O})$ comes from the fact that $T \mapsto M(T)$ is not an exact functor. However, as the following lemma shows, this functor does become exact after inverting $p$.
\begin{lemma}\label{exact}
	If $0 \rightarrow T_1 \rightarrow T \rightarrow T_2 \rightarrow 0$ is an exact sequence of crystalline $\mathbb{Z}_p$-lattices then $0 \rightarrow M(T_1)[\frac{1}{p}] \rightarrow M(T)[\frac{1}{p}] \rightarrow M(T_2)[\frac{1}{p}] \rightarrow 0$ is an exact sequence of $\mathfrak{S}[\frac{1}{p}]$-modules.
\end{lemma}
\begin{proof}
	For this we must use that the functor in Theorem~\ref{kisin} satisfies an additional property; namely if $T$ is a crystalline $\mathbb{Z}_p$-lattice inside $V = T[\frac{1}{p}]$ then there exists a $\varphi$-equivariant identification 
	$$
	M(T) \otimes_{\mathfrak{S}} \mathcal{O}^{\operatorname{rig}}[\tfrac{1}{\lambda}] \cong D_{\operatorname{crys}}(V) \otimes_{K_0} \mathcal{O}^{\operatorname{rig}}[\tfrac{1}{\lambda}]
	$$
	This is a consequence of \cite[Lemma 1.2.6]{Kis06}. Here $\mathcal{O}^{\operatorname{rig}} \subset K_0[[u]]$ is the subring of power series which converge on the open unit disk, and $\lambda \in \operatorname{\mathcal{O}}^{\operatorname{rig}}$ is the convergent product $\prod_{n=0}^\infty \varphi^n(E(u)/E(0))$. Since $V \mapsto D_{\operatorname{crys}}(V)$ is an exact functor it suffices to show that $\mathfrak{S}[\frac{1}{p}]$ is faithfully flat over $\mathcal{O}^{\operatorname{rig}}[\frac{1}{\lambda}]$. Since $\mathfrak{S}[\frac{1}{p}]$ is a principal ideal domain faithful flatness follows because if $f \in \mathfrak{S}[\frac{1}{p}]$ is not a unit then $f$ is not a unit in $\mathcal{O}^{\operatorname{rig}}[\frac{1}{\lambda}]$ (if it was then $f$ must have either no zeroes on the open unit disk, or infinitely many, at the zeroes of $\lambda$).
\end{proof}

Consider the category $\operatorname{Mod}^{\operatorname{BK-iso}}_K(\mathcal{O})$ of $\mathfrak{S}_E :=\mathfrak{S} \otimes_{\mathbb{Z}_p} E$-modules $M$ equipped with isomorphisms $\varphi_M\colon M \otimes_{\varphi,\mathfrak{S}_E} \mathfrak{S}_E[\frac{1}{E}] \cong M[\frac{1}{E}]$, such that $\varphi$-equivariantly $M \cong M^\circ[\frac{1}{p}]$ for some $M^\circ \in \operatorname{Mod}^{\operatorname{BK}}_K(\mathcal{O})$. This category can be identified with the isogeny category of $\operatorname{Mod}^{\operatorname{BK}}_K(\mathcal{O})$, in particular it is abelian. 

By \cite[Proposition 4.3]{BMS} every object of $\operatorname{Mod}^{\operatorname{BK-iso}}_K(\mathcal{O})$ is free as an $\mathfrak{S}$-module; arguing as in Corollary~\ref{freecoeff} we see they are free as $\mathfrak{S} \otimes_{\mathbb{Z}_p} E$-modules.

\begin{corollary}\label{isofunctor}
	The functor $T \mapsto M(T)$ induces an exact fully faithful functor $V \mapsto M(V)$ from the category of crystalline $E$-representations to $\operatorname{Mod}^{\operatorname{BK-iso}}_K(\mathcal{O})$.
\end{corollary}
\begin{proof}
	This follows from Lemma~\ref{exact}, and the fact that $T \mapsto M(T)$ is fully faithful.
\end{proof}

\subsection{More ext groups} As in Subsection~\ref{cohom}, if $M \in \operatorname{Mod}^{\operatorname{BK}}_K(\mathcal{O})$ we can define $H^1(M)$ as the cokernel of $M \xrightarrow{\varphi -1} M[\frac{1}{E}]$. Arguing as in Construction~\ref{H1andExt}, if $P \in \operatorname{Mod}^{\operatorname{BK}}_K(\mathcal{O})$ also there is a functorial inclusion 
$$
H^1(\operatorname{Hom}(P,M)^{\mathcal{O}}) \hookrightarrow \operatorname{Ext}^1_{\mathcal{O}}(P,M)
$$ 
where $\operatorname{Ext}^1_{\mathcal{O}}$ denote the Yoneda extension group in the abelian category $\operatorname{Mod}^{\operatorname{BK}}_K(\mathcal{O})$. In particular this is a map of $\mathcal{O}$-modules. This map is surjective if $P$ is projective as an $\mathfrak{S}_{\mathcal{O}}$-module (for then every extension of $P$ by $M$ splits as an $\mathfrak{S}_{\mathcal{O}}$-module).

Suppose $P$ and $M$ are free $\mathfrak{S}_{\mathcal{O}}$-modules, so that $\overline{P} = P \otimes_{\mathcal{O}} \mathbb{F}$ and $\overline{M} = M \otimes_{\mathcal{O}} \mathbb{F}$ are objects of $\operatorname{Mod}^{\operatorname{BK}}_k(\mathcal{O})$. There is a map 
\begin{equation}\label{modp}
\operatorname{Ext}^1_{\mathcal{O}}(P,M) = H^1(\operatorname{Hom}(P,M)^{\mathcal{O}}) \rightarrow H^1(\operatorname{Hom}(\overline{P},\overline{M})^{\mathcal{O}}) = \operatorname{Ext}^1_{\mathbb{F}}(\overline{P},\overline{M})
\end{equation}
induced by $\operatorname{Hom}(P,M)^{\mathcal{O}} \rightarrow \operatorname{Hom}(\overline{P},\overline{M})^{\mathcal{O}}$. On the level of exact sequences this map is given by tensoring with $\mathbb{F}$ over $\mathcal{O}$. Via $\operatorname{Hom}(P,M)^{\mathcal{O}} \rightarrow \operatorname{Hom}(\overline{P},\overline{M})^{\mathcal{O}}$ we can identify $\operatorname{Hom}(P,M)^{\mathcal{O}} \otimes_{\mathcal{O}} \mathbb{F} = \operatorname{Hom}(\overline{P},\overline{M})^{\mathcal{O}}$. Therefore, since
$$
\operatorname{Hom}(P,M)^{\mathcal{O}} \otimes_{\mathcal{O}} \mathbb{F} \rightarrow \operatorname{Hom}(P,M)^{\mathcal{O}}[\tfrac{1}{E}] \otimes_{\mathcal{O}} \mathbb{F} \rightarrow H^1(\operatorname{Hom}(P,M)^{\mathcal{O}}) \otimes_{\mathcal{O}} \mathbb{F} \rightarrow 0
$$
is exact, it follows that via \eqref{modp} we can identify $\operatorname{Ext}^1_{\mathcal{O}}(P,M) \otimes_{\mathcal{O}} \mathbb{F} = \operatorname{Ext}^1_{\mathbb{F}}(\overline{P},\overline{M})$.

Analogously, for $M \in \operatorname{Mod}^{\operatorname{BK-iso}}_K(\mathcal{O})$ we define $H^1(M)$ as the cokernel of $\varphi -1 \colon M \rightarrow M[\frac{1}{E}]$. Just as in Construction~\ref{H1andExt}, if $P \in \operatorname{Mod}^{\operatorname{BK-iso}}_K(\mathcal{O})$ there are inclusions 
$$
H^1(\operatorname{Hom}(P,M)^E) \hookrightarrow \operatorname{Ext}^1_E(P,M)
$$
where $\operatorname{Ext}^1_E(P,M)$ denotes the Yoneda extension group in $\operatorname{Mod}^{\operatorname{BK-iso}}_K(\mathcal{O})$. Since $P$ is free as an $\mathfrak{S}_E$-module this inclusion is an isomorphism. Choose $P^\circ,M^\circ \in \operatorname{Mod}^{\operatorname{BK}}_K(\mathcal{O})$ such that $P^\circ \otimes_{\mathcal{O}} E = P,M^\circ \otimes_{\mathcal{O}} E = M$. Suppose $P^\circ$ and $M^\circ$ are free over $\mathfrak{S}_{\mathcal{O}}$, then the inclusion $\operatorname{Hom}(P,M)^{\mathcal{O}} \hookrightarrow \operatorname{Hom}(P,M)^E$ induces maps
\begin{equation}\label{beta}
\operatorname{Ext}^1_{\mathcal{O}}(P^\circ,M^\circ) = H^1(\operatorname{Hom}(P^\circ,M^\circ)^{\mathcal{O}}) \rightarrow H^1(\operatorname{Hom}(P,M)^E) = \operatorname{Ext}^1_E(P,M)
\end{equation}
On the level of exact sequences this map is given by applying $\otimes_{\mathcal{O}} E$. Similarly to above \eqref{beta} induces identifications $\operatorname{Ext}^1_{\mathcal{O}}(P^\circ,M^\circ) \otimes_{\mathcal{O}} E = \operatorname{Ext}^1_E(P,M)$.

\subsection{Lifting extensions} Our aim is to prove the following.

\begin{proposition}\label{liftingextThm}
	Suppose $K = K_0$ and if $p =2$ suppose further that $\pi$ is chosen so that $K_\infty \cap K(\mu_{p^\infty}) = K$. Let $T_2,T_1$ be crystalline $\mathcal{O}$-lattices with Hodge--Tate weights contained in $[0,p]$. Set $\overline{T}_i = T_i \otimes_{\mathcal{O}} \mathbb{F}$. Assume $\overline{T}_1$ is cyclotomic-free, that $\overline{T}_2$ is irreducible, and that $\overline{T}_2 \otimes_{\mathbb{F}} \mathbb{F}(1)$ is not a Jordan--Holder factor of $\overline{T}_1$.
	
	Set $\overline{M}_i = M(T_i) \otimes_{\mathcal{O}} \mathbb{F}$ and suppose 
	\begin{equation}\label{tobelifted}
	0 \rightarrow \overline{M}_1\rightarrow \overline{M} \rightarrow \overline{M}_2 \rightarrow 0
	\end{equation}
	is an exact sequence in $\operatorname{Mod}^{\operatorname{SD}}_k(\mathcal{O})$. Then there exists a crystalline extension $0 \rightarrow T_1 \rightarrow T \rightarrow T_2 \rightarrow 0$ such that $0 \rightarrow M(T_1) \rightarrow M(T) \rightarrow M(T_2) \rightarrow 0$ is exact and recovers \eqref{tobelifted} after applying $\otimes_{\mathcal{O}} \mathbb{F}$.
\end{proposition}
\begin{proof}
	Let $V_i = T_i \otimes_{\mathcal{O}} E$ and let $\operatorname{Ext}^1_{\operatorname{crys}}(V_2,V_1) \subset \operatorname{Ext}^1(V_2,V_1)$ denote the subset whose elements are represented by exact sequences $0 \rightarrow V_1 \rightarrow V \rightarrow V_2 \rightarrow 0$ with $V$ crystalline. Under the usual identification $\operatorname{Ext}^1(V_2,V_1) = H^1(G_K,\operatorname{Hom}(V_2,V_1))$, the subspace $\operatorname{Ext}^1_{\operatorname{crys}}(V_2,V_1)$ identifies with 
	$$
	H^1_f(G_K,\operatorname{Hom}(V,W)) = \operatorname{ker}\Big( H^1(G_K,\operatorname{Hom}(V_2,V_1)) \rightarrow H^1(G_K,\operatorname{Hom}(V_2,V_1) \otimes_{\mathbb{Q}_p} B_{\operatorname{crys}}) \Big)
	$$
	Since the Hodge--Tate weights of $\operatorname{Hom}(V_2,V_1)$ are equal to $i  -j$ where $i$ is a weight of $V_1$ and $j$ a weight of $V_2$, \cite[Proposition 1.24]{Nek93} implies $H^1_f(G_K,\operatorname{Hom}(V_2,V_1))$ has $E$-dimension
	$$
	\sum_{\tau \in \operatorname{Hom}_{\mathbb{F}_p}(k,\mathbb{F})} \operatorname{Card}(\lbrace i -j < 0 \mid i \in \operatorname{HT}_\tau(V_1), j \in \operatorname{HT}_\tau(V_2) \rbrace) + \operatorname{dim}_E \operatorname{Hom}_{E[G_K]}(V_2,V_1)
	$$
	
	Now consider the diagram
	$$
	\begin{tikzcd}[column sep =small, row sep=small]
	~& \operatorname{Ext}^1_{\mathcal{O}}(M(T_2),M(T_1)) \arrow{r}{\beta} \arrow{d}{\alpha} & \operatorname{Ext}^1_{\mathbb{F}}(\overline{M}_2,\overline{M}_1) \\
	\operatorname{Ext}^1_{\operatorname{crys}}(V_2,V_1) \arrow{r}{\gamma} & \operatorname{Ext}^1_E(M(V_2),M(V_1)) & ~ 
	\end{tikzcd}
	$$
	where the maps $\alpha$ and $\beta$ are those described in  \eqref{beta} and \eqref{modp} respectively. The map $\gamma$ is obtained by applying $V \mapsto M(V)$ (from Corollary~\ref{isofunctor}) to exact sequences representing classes in $\operatorname{Ext}_{\operatorname{crys}}(V_2,V_1)$. This makes sense since $V \mapsto M(V)$ is exact. Exactness of $M \mapsto M(V)$ also implies that this functor preserves pushouts and pullbacks; thus $\gamma$ is $E$-linear. Since $V \mapsto M(V)$ is fully faithful we see that $\gamma$ is injective.
	
	Let $\Theta'$ denote the image of $\gamma$ and let $\Theta$ denote the preimage of $\Theta'$ under $\alpha$. If $0 \rightarrow M(T_1) \rightarrow M^\circ \rightarrow M(T_2) \rightarrow 0$ represents a class in $\Theta$ then by definition $M^\circ \otimes_{\mathcal{O}} E = M(V)$ where $V$ is a crystalline $E$-representation fitting into an extension $0 \rightarrow V_1 \rightarrow V \rightarrow V_2 \rightarrow 0$. Thus $T =T(M^\circ)$ is a $G_{K_\infty}$-stable $\mathcal{O}$-lattice inside $V$ which, since $M \mapsto T(M)$ is exact, sits in a $G_{K_\infty}$-equivariant exact sequence $0 \rightarrow T_1 \rightarrow T \rightarrow T_2 \rightarrow 0$. Our assumption on $\overline{T}_2$ and $\overline{T}_1$ allows us to apply Corollary~\ref{stable}; thus $T$ is a $G_{K}$-stable lattice in $V$ and so $M^\circ = M(T)$. Theorem~\ref{GLS} therefore implies $\beta$ maps every element of $\Theta$ into $\operatorname{Ext}^1_{\operatorname{SD}}(\overline{M}_2,\overline{M}_1)$. 
	
	The map $\beta$ can be identified with $\operatorname{Ext}^1_{\mathcal{O}}(M(T_2),M(T_1)) \rightarrow \operatorname{Ext}^1_{\mathcal{O}}(M(T_2),M(T_1) \otimes_{\mathcal{O}} \mathbb{F}$. Since $\alpha x \in \Theta$ for any $\alpha \in \mathcal{O}$ implies $x \in \Theta$, the cokernel of $\Theta \subset \operatorname{Ext}^1_{\mathcal{O}}(M(T_2),M(T_1))$ is free over $\mathcal{O}$ and so $\Theta \hookrightarrow \operatorname{Ext}^1_{\mathcal{O}}(M(T_2),M(T_1))$ induces an inclusion $\Theta \otimes_{\mathcal{O}} \mathbb{F} \hookrightarrow \operatorname{Ext}^1_{\mathcal{O}}(M(T_2),M(T_1)) \otimes_{\mathcal{O}} \mathbb{F}$. As $\beta$ maps every element of $\Theta$ into $\operatorname{Ext}^1_{\operatorname{SD}}(\overline{M_2},\overline{M}_1)$, we have
	$$
	\Theta \otimes_{\mathcal{O}} \mathbb{F} \hookrightarrow \operatorname{Ext}^1_{\operatorname{SD}}(\overline{M}_2,\overline{M}_1)
	$$
	On the other hand, since $\alpha$ is given by inverting $p$, the image of $\alpha$ is an $\mathcal{O}$-lattice inside $\operatorname{Ext}^1_E(M(V_2),M(V_1))$ and its kernel is the torsion subgroup $\operatorname{Ext}^1_{\mathcal{O}}(M(T_2),M(T_1))_{\operatorname{tors}}$. Thus we can decompose $\Theta$ as
	$$
	\Theta_{\operatorname{free}} \bigoplus \operatorname{Ext}_{\mathcal{O}}^1(M(T_2),M(T_1))_{\operatorname{tors}}
	$$ 
	where $\Theta_{\operatorname{free}}$ is a free $\mathcal{O}$-module of rank equal to the $E$-dimension of $\operatorname{Ext}^1_{\operatorname{crys}}(V_2,V_1)$. Using Lemma~\ref{below} below and formula for the dimension of $\operatorname{Ext}^1_{\operatorname{crys}}(V_2,V_1)= H^1_f(G_K,\operatorname{Hom}(V_2,V_1))$ above, we deduce
	$$
	\begin{aligned}
	\operatorname{dim}_{\mathbb{F}}(\Theta \otimes_{\mathcal{O}}\mathbb{F}) - &\operatorname{dim}_{\mathbb{F}} \operatorname{Hom}_{\operatorname{BK}}(\overline{M}_2,\overline{M}_1) = \\
	&\sum_\tau \operatorname{Card}(\lbrace i - j < 0 \mid i \in \operatorname{HT}_{\tau}(V_1),j \in \operatorname{HT}_\tau(V_2) \rbrace) 
	\end{aligned}
	$$
	By Theorem~\ref{GLS} we have $\operatorname{HT}_\tau(V_i) = \operatorname{Weight}_\tau(\overline{M}_i)$. Therefore, by Proposition~\ref{DIMFORM}, each of $\operatorname{Ext}^1_{\operatorname{SD}}(\overline{M}_2,\overline{M}_1)$ and $\Theta \otimes_{\mathcal{O}} \mathbb{F}$ have the same $\mathbb{F}$-dimension. Hence 
	$$
	\Theta \otimes_{\mathcal{O}} \mathbb{F} = \operatorname{Ext}^1_{\operatorname{SD}}(\overline{M}_2,\overline{M}_1)
	$$
	which shows that any extension $0 \rightarrow \overline{M}_1 \rightarrow M \rightarrow \overline{M}_2 \rightarrow 0$ which represents a class in $\operatorname{Ext}^1_{\operatorname{SD}}(\overline{M}_1,\overline{M}_2)$ arises as the reduction of $0 \rightarrow M(T_1) \rightarrow M(T) \rightarrow M(T_2) \rightarrow 0$ for some crystalline extension $0 \rightarrow T_1 \rightarrow T \rightarrow T_2 \rightarrow 0$.
\end{proof}
\begin{lemma}\label{below}
	If $T_1$ and $T_2$ are crystalline $\mathcal{O}$-lattices then
	\begin{gather*}
	\operatorname{dim}_{\mathbb{F}}( \operatorname{Ext}^1_{\mathcal{O}}(M(T_2),M(T_1))_{\operatorname{tors}} \otimes_{\mathcal{O}} \mathbb{F}) = \operatorname{dim}_{\mathbb{F}} \operatorname{Hom}_{\operatorname{BK}}(\overline{M}_2,\overline{M}_1)\\
	- \operatorname{dim}_{E} \operatorname{Hom}_{E[G_K]}(V_2,V_1)
	\end{gather*}
\end{lemma}
\begin{proof}
	The $\mathbb{F}$-dimension of $\operatorname{Ext}_{\mathcal{O}}^1(M(T_2),M(T_1))_{\operatorname{tors}} \otimes_{\mathcal{O}} \mathbb{F}$ equals the $\mathbb{F}$-dimension of the $\varpi$-torsion subgroup of $\operatorname{Ext}^1_{\mathcal{O}}(M(T_2),M(T_1))$. To compute this latter group consider the exact sequence $0 \rightarrow M(T_1) \xrightarrow{\varpi} M(T_1) \rightarrow \overline{M}_1 \rightarrow 0$ in $\operatorname{Mod}^{\operatorname{BK}}_K(\mathcal{O})$; the associated long exact sequence reads
	$$
	\begin{aligned}
	0 \rightarrow \operatorname{Hom}_{\operatorname{BK}}&(M(T_2),M(T_1)) \xrightarrow{\varpi} \operatorname{Hom}_{\operatorname{BK}}(M(T_2),M(T_1)) \rightarrow \operatorname{Hom}_{\operatorname{BK}}(M(T_2),\overline{M}_1) \\
	&\rightarrow \operatorname{Ext}^1_{\mathcal{O}}(M(T_2),M(T_1)) \xrightarrow{\varpi} \operatorname{Ext}^1_{\mathcal{O}}(M(T_2),M(T_1))
	\end{aligned}
	$$
	Identifying $\operatorname{Hom}_{\operatorname{BK}}(M(T_2),\overline{M}_1) = \operatorname{Hom}_{\operatorname{BK}}(\overline{M}_2,\overline{M}_1)$ we see that the $\mathbb{F}$-dimension of the $\varpi$-torsion subgroup of $\operatorname{Ext}^1_{\mathcal{O}}(M(T_2),M(T_1))$ equals
	$$
	\operatorname{dim}_{\mathbb{F}} \operatorname{Hom}_{\operatorname{BK}}(\overline{M}_2,\overline{M}_1) - \operatorname{dim}_{\mathbb{F}} \operatorname{Hom}_{\operatorname{BK}}(M(T_2),M(T_1))/\varpi
	$$
	By full faithfulness of $T \mapsto M(T)$ we have that $\operatorname{Hom}_{\operatorname{BK}}(M(T_2),M(T_1)) = \operatorname{Hom}_{\mathcal{O}[G_{K}]}(T_2,T_1)$. Since $\operatorname{Hom}_{\mathcal{O}[G_K]}(T_2,T_1)$ is $\mathcal{O}$-free and equals $\operatorname{Hom}_{E[G_K]}(V_2,V_1)$ after inverting $p$ the lemma follows.
\end{proof}
\section{Lifting irreducibles}\label{section6}

In this section we study simple objects of $\operatorname{Mod}^{\operatorname{SD}}_k(\mathcal{O})$ in low dimensions and show they arise from crystalline representations.

\subsection{Rank ones}
Recall from Construction~\ref{O-semilinear} how $\mathfrak{S} \otimes_{\mathbb{Z}_p} \mathcal{O}$ is made into an $\mathcal{O}[[u]]$-algebra. In this way we view $k[[u]] \otimes_{\mathbb{F}_p} \mathbb{F}$ as an $\mathbb{F}[[u]]$-algebra. Also let $e_\tau \in k[[u]] \otimes_{\mathbb{F}_p} \mathbb{F}$ denote the image of the idempotent $\widetilde{e}_\tau \in \mathfrak{S} \otimes_{\mathbb{Z}_p} \mathcal{O}$ defined in Construction~\ref{O-semilinear}.

\begin{lemma}
After possibly enlarging $\mathbb{F}$, every $M \in \operatorname{Mod}^{\operatorname{BK}}_k(\mathcal{O})$ of rank one over $k[[u]] \otimes_{\mathbb{F}_p} \mathbb{F}$ is isomorphic to a Breuil--Kisin module
\begin{equation}\label{rk1}
N = k[[u]] \otimes_{\mathbb{F}_p} \mathbb{F}, \qquad \varphi_N(1) = x\sum_{\tau \in \operatorname{Hom}_{\mathbb{F}_p}(k,\mathbb{F})} u^{r_{\theta}}e_{\theta} 
\end{equation}
for some $r_{\theta} \in \mathbb{Z}$ and some $x \in \mathbb{F}^\times$. 
\end{lemma}
\begin{proof}
	For any $f \in 1+ u\mathbb{F}[[u]]$ and $n \geq 0$ it is easy to check th equation $\varphi^n(z) = fz$ has a solution in $\mathbb{F}[[u]]^\times$ (here $\varphi$ on $\mathbb{F}[[u]]$ denotes the Frobenius $\sum a_iu^i \mapsto \sum a_iu^{ip}$). Thus, after possibly enlarging $\mathbb{F}$, if $f \in \mathbb{F}[[u]]^\times$ there exists $z \in \mathbb{F}[[u]]^\times$ and $x \in \mathbb{F}$ such that $\varphi^n(z) = x^n f z$.

	Now consider the statement of the lemma. For each $\tau$ choose a generator $x_{\tau}$ of $M_\tau$ over $\mathbb{F}[[u]]$. There are integers $r_\tau$ and $f_\tau \in \mathbb{F}[[u]]^\times$ such that $\varphi(x_{\tau \circ \varphi}) = u^{r_\tau} f_\tau e_\tau$. Recall that $\varphi_M \colon M_{\tau \circ \varphi}\rightarrow M_\tau[\tfrac{1}{u}]$ is semi-linear for the endomorphism $\varphi$ of $\mathbb{F}[[u]]$ described in the previous paragraph. Choose $z \in \mathbb{F}[[u]]^\times$ and $x \in \mathbb{F}^\times$ so that
	\begin{equation}\label{iterate}
	\frac{\varphi^{[k:\mathbb{F}_p]}(z)}{z} = x^{[k:\mathbb{F}_p]}\Bigg(  f_{\tau} \varphi(f_{\tau \circ \varphi^{-1}}) \ldots \varphi^{[k:\mathbb{F}_p]-1}(f_{\tau \circ \varphi^{-[k:\mathbb{F}_p]-1}}) \Bigg)^{-1}
	\end{equation}
	Set $y_{-1} = zx_{\tau \circ \varphi^{-1}}$ and for $i>1$ set 
	$$
	y_{-i} = \frac{\varphi^{i-1}(z)}{x^{i-1}} f_{\tau \circ \varphi^{-i}} \varphi(f_{\tau \circ \varphi^{-i+1}}) \ldots \varphi^{i-2}(f_{\tau \circ \varphi^{-2}})x_{\tau \circ \varphi^{-i}}
	$$
	Then $y_{-i}$ generates $M_{\theta \circ \varphi^{-i}}$. By construction $\varphi(y_{-i+1}) = x y_{-i}$ for $2 \leq i \leq [k:\mathbb{F}_p]$, also $\varphi(y_{-[k:\mathbb{F}_p]}) = x y_{-1}$ because $x$ and $z$ satisfy \eqref{iterate}. Thus the map $M \rightarrow N$ sending $y_{-i}$ onto $e_{\tau \circ \varphi^{-i}}$ is an isomorphism of Breuil--Kisin modules.
\end{proof}
\begin{proposition}\label{crystliftrk1}
	Assume $K = K_0$ and $N \in \operatorname{Mod}^{\operatorname{BK}}_k(\mathcal{O})$ is as in \eqref{rk1}. Then there exists a rank one crystalline $\mathcal{O}$-lattice $T$ with $\tau$-Hodge--Tate weight $r_\tau$ and such that $M(T) \otimes_{\mathcal{O}} \mathbb{F} \cong N$.
\end{proposition}
\begin{proof}
	This is proven in \cite[Lemma 6.3]{GLS}.
\end{proof}
\subsection{Induction and restriction}\label{indres} Let $L/K$ be the unramified extension corresponding to a finite extension $l/k$, and let $L_\infty = K_\infty L$. Set $\mathfrak{S}_L = W(l)[[u]]$. Extension of scalars along the inclusion $f\colon \mathfrak{S} \rightarrow \mathfrak{S}_L$ describes a functor
	$$
	f^*\colon \operatorname{Mod}^{\operatorname{BK}}_K \rightarrow \operatorname{Mod}^{\operatorname{BK}}_L
	$$
	For $M \in \operatorname{Mod}^{\operatorname{BK}}_K$ the module $f^*M = M \otimes_{\mathfrak{S}} \mathfrak{S}_L$ is made into a Breuil--Kisin module via the semilinear map $m \otimes s \mapsto \varphi_M(m)\otimes \varphi(s)$. Similarly restriction of scalars along $f$ induces a functor
	$$
	f_*\colon \operatorname{Mod}^{\operatorname{BK}}_L \rightarrow \operatorname{Mod}^{\operatorname{BK}}_K
	$$
	If $M \in \operatorname{Mod}^{\operatorname{BK}}_L$ we equip $f_*M$ with the obvious semilinear map $m \mapsto \varphi_M(m)$.

\begin{lemma}\label{yoneda}
	Let $N \in \operatorname{Mod}^{\operatorname{BK}}_L$ and $M \in \operatorname{Mod}^{\operatorname{BK}}_K$. Then there are functorial identifications $\operatorname{Hom}_{\operatorname{BK}}(M,f_*N) \cong \operatorname{Hom}_{\operatorname{BK}}(f^*M,N)$, $T(f^*M) \cong T(M)|_{G_{L_\infty}}$, and $T(f_*N) \cong \operatorname{Ind}^{K_\infty}_{L_\infty} T(N)$ making the following diagram commute.
	$$
	\begin{tikzcd}[column sep =small, row sep=small]
	\operatorname{Hom}_{\operatorname{BK}}(M,f_*N) \arrow{r} \arrow{d}{T} & \operatorname{Hom}_{\operatorname{BK}}(f^*M,N) \arrow{d}{T}\\
	\operatorname{Hom}_{G_{K_\infty}}(T(M),\operatorname{Ind}_{L_\infty}^{K_\infty} T(N)) \arrow{r}{(\operatorname{Frob})} & \operatorname{Hom}_{G_{L_\infty}}(T(M)|_{G_{L_\infty}},T(N))
	\end{tikzcd}
	$$
	The lower horizontal arrow is given by Frobenius reciprocity.
\end{lemma}
\begin{proof}
	This is \cite[Lemma 6.2.4]{B18}.
\end{proof}

By functoriallity $f_*$ and $f^*$ induce functors on $\operatorname{Mod}^{\operatorname{BK}}_K(\mathcal{O})$ and $\operatorname{Mod}^{\operatorname{BK}}_L(\mathcal{O})$. The following lemma shows how they also preserve strong divisibility.
\begin{lemma}\label{SDinductionrestriction}
	Assume $k \subset l \subset \mathbb{F}$.
	\begin{enumerate}
		\item If $M \in \operatorname{Mod}^{\operatorname{SD}}_k(\mathcal{O})$ then $f^*M \in \operatorname{Mod}^{\operatorname{SD}}_l(\mathcal{O})$ and for each $\theta \in \operatorname{Hom}_{\mathbb{F}_p}(l,\mathbb{F})$ we have 
		$$
		\operatorname{Weight}_\theta(f^*M) = \operatorname{Weight}_{\theta|_k}(M)
		$$
		\item If $N \in \operatorname{Mod}^{\operatorname{SD}}_l(\mathcal{O})$ then $f_*N \in \operatorname{Mod}^{\operatorname{SD}}_k(\mathcal{O})$ and 
		$$
		\operatorname{Weight}_\tau(f_*N) = \bigcup_{\theta|_k = \tau} \operatorname{Weight}_\theta(N)
		$$
	\end{enumerate} 
\end{lemma}
\begin{proof}
	This is \cite[Lemma 6.2.6]{B18}.
\end{proof}

\subsection{Approximation by rank ones}\label{approx} For this subsection fix an irreducible $M \in \operatorname{Mod}^{\operatorname{SD}}_k(\mathcal{O})$. Irreducible here means the only non-trivial sub-Brueil--Kisin module $M' \subset M$ in $\operatorname{Mod}^{\operatorname{BK}}_k(\mathcal{O})$ with torsion-free cokernel is $M$ itself. Let $L/K$ be the unramified extension of degree $\operatorname{dim}_{\mathbb{F}} T(M)$ with residue extension $l/k$, and let $f\colon \mathfrak{S} \rightarrow \mathfrak{S}_L$ be as in Subsection~\ref{indres}.

\begin{proposition}\label{approximate}
	After possibly enlarging $\mathbb{F}$ there exists a rank one $N \in \operatorname{Mod}^{\operatorname{SD}}_l(\mathcal{O})$ and an inclusion $M \rightarrow f_*N$ whose image contains $u(f_*N)$. 
	
	Further, if for $\theta \in \operatorname{Hom}_{\mathbb{F}_p}(l,\mathbb{F})$ we set $\delta_\theta = 0$ if the image of $M$ contains $N_\theta$, and $\delta_\theta =1$ otherwise, then
	$$
	r_{\theta} , r_{\theta} + p \delta_{\theta \circ \varphi} - \delta_\theta \in \operatorname{Weight}_{\theta|_k}(M)
	$$
	where $\lbrace r_\theta \rbrace = \operatorname{Weight}_{\theta}(N)$.
\end{proposition}

The construction is identical to that given in \cite[\S6.3]{B18}. For the convenience of the reader we repeat the arguments. After Lemma~\ref{rk1} we may describe $N$ explicitly as
$$
N = k[[u]] \otimes_{\mathbb{F}_p} \mathbb{F}, \qquad \varphi_N(1) = x\sum u^{r_\theta} e_\theta
$$
for some $x \in \mathbb{F}^\times$, with the sum running over $\theta \in \operatorname{Hom}_{\mathbb{F}_p}(l,\mathbb{F})$.

\begin{proof}
First we show that $M$ being irreducible implies $T(M)$ is irreducible as a $G_{K_\infty}$-representation. This follows from the next lemma.

\begin{lemma}\label{exactly}
	Let $M \in \operatorname{Mod}^{\operatorname{BK}}_k(\mathcal{O})$ and let $0 \rightarrow T_1 \rightarrow T(M) \rightarrow T_2 \rightarrow 0$ be an exact sequence of $G_{K_\infty}$-representations. Then there exists an exact sequence $0 \rightarrow M_1 \rightarrow M \rightarrow M_2 \rightarrow 0$ in $\operatorname{Mod}^{\operatorname{BK}}_k(\mathcal{O})$ which recovers the first exact sequence after applying $M \mapsto T(M)$.
\end{lemma}
\begin{proof}
	Since $T(M) \otimes_{\mathbb{Z}_p} W(C^\flat) = M \otimes_{\mathfrak{S}} W(C^\flat)$ we obtain a surjection $M \otimes_{\mathfrak{S}} W(C^\flat) \rightarrow T_2 \otimes_{\mathbb{Z}_p} W(C^\flat)$. Let $M_2$ be the image of $M$ under this surjection. Then $M_2 \otimes_{\mathfrak{S}} W(C^\flat) = T_2 \otimes_{\mathbb{Z}_p} W(C^\flat)$ so $T(M_2) = T_2$. Take $M_1$ to be the kernel of $M \rightarrow M_2$.
\end{proof}

Enlarging $\mathbb{F}$ if necessary we can suppose $l \subset \mathbb{F}$ and that $T(M) \cong \operatorname{Ind}_{L_\infty}^{K_\infty} \chi$ (using Lemma~\ref{irredinduced} and Lemma~\ref{equiv}) for a character $\chi$. Frobenius reciprocity gives a non-zero map $T(M)|_{G_{L_\infty}} \rightarrow \chi$. Lemma~\ref{yoneda} implies $T(M)|_{G_{L_\infty}} \cong T(f^*M)$, so Lemma~\ref{exactly} produces a surjection $f^*M \rightarrow N$ for a rank one $N \in \operatorname{Mod}^{\operatorname{BK}}_l(\mathcal{O})$ with $T(N) = \chi$. Lemma~\ref{yoneda} then implies there exists a map 
$$
M \rightarrow f_*N
$$
which, after applying $T$, induces the isomorphism $T(M) \cong \operatorname{Ind}_{L_\infty}^{K_\infty} \chi$. In other words $M \rightarrow f_*N$ becomes an isomorphism after inverting $u$; and so is injective.

Lemma~\ref{SDinductionrestriction} and Proposition~\ref{SDexact} imply $N \in \operatorname{Mod}_{l}^{\operatorname{SD}}(\mathcal{O})$ and $\operatorname{Weight}_\theta(N) \subset \operatorname{Weight}_\theta(f^*M) = \operatorname{Weight}_{\theta|_k}(M)$. 

Now let us assume that $N$ is as in Lemma~\ref{rk1}, so that $N$ is generated over $\mathbb{F}[[u]]$ by the $e_\theta$, with $\varphi_N(e_{\theta \circ \varphi}) = xu^{r_{\theta}} e_\theta$. Since $M \rightarrow f_*N$ is an isomorphism after inverting $u$ we may consider the smallest integer $\delta_\theta \geq 0$ satisfying $u^{\delta_\theta} e_\theta \in M$. We shall show that $\delta_\theta \in [0,1]$ which shows $\delta_\theta$ is equal to the $\delta_\theta$ defined in the proposition. Let $P \in \operatorname{Mod}^{\operatorname{BK}}_l(\mathcal{O})$ be the sub-Breuil--Kisin module of $N$ generated by the $u^{\delta_\theta}e_\theta$. The map $P \rightarrow f^*M$ given by $ u^{\delta_\theta}e_\theta \mapsto e_\theta( u^{\delta_\theta} e_\theta \otimes 1)$ is $\varphi$-equivariant and has $u$-torsion-free cokernel, by definition of the $\delta_\theta$. Therefore Proposition~\ref{SDexact} implies $P \in \operatorname{Mod}^{\operatorname{SD}}_l(\mathcal{O})$ and $\operatorname{Weight}_\theta(P) \subset \operatorname{Weight}_\theta(f^*M) = \operatorname{Weight}_{\theta|_k}(M)$. Since $\operatorname{Weight}_\theta(P) = \lbrace r_\theta + p \delta_{\theta \circ \varphi} - \delta_\theta \rbrace$ it just remains to prove $\delta_{\theta} \in [0,1]$.

Since $r_{\theta},r_{\theta} + p\delta_{\theta \circ \varphi} - \delta_\theta$ are both contained in $\operatorname{Weight}_{\theta|_k}(M)$ both are in $[0,p]$. Thus $p \delta_{\theta \circ \varphi} - \delta_\theta  \leq p$ and so
$$
(p^{[l:\mathbb{F}_p]} - 1)\delta_\theta = \sum_{i=1}^{[l:\mathbb{F}_p]} p^{i-1}( p\delta_{\theta \circ \varphi^i} - \delta_{\theta \circ \varphi^{i-1}}) \leq p(p^{[l:\mathbb{F}_p]} -1)/(p-1)
$$
which shows $\delta_{\theta} \in [0,1]$ unless $p=2$, in which case we deduce $\delta_{\theta} \in [0,2]$. If $p=2$ and $\delta_{\theta \circ \varphi} = 2$ then as $r_{\theta} + p\delta_{\theta \circ \varphi} - \delta_\theta \in [0,p]$ it follows that $r_\theta = 0$ and $\delta_{\theta} = 2$. Thus if $\delta_\theta \not\in [0,1]$ for some $\theta$ then $r_\theta = 0$ for all $\theta$; this implies $T(N)$ is an unramified character and so $\operatorname{Ind}_{L_\infty}^{K_\infty} T(N)$ is not irreducible, a contradiction.
\end{proof}

\begin{remark}\label{surjec}
	The surjection $f^*M \rightarrow N$ can be recovered from the inclusion $M \rightarrow f_*N$ explicitly. On $(f^*M)_\theta = M_{\theta|_k}$ it is given by $\sum_{\theta'|_k = \theta|_k} \alpha_{\theta'} e_{\theta'} \mapsto \alpha_\theta e_\theta$. In particular, for every $\theta$ there exist $\alpha_{\theta'} \in \mathbb{F}[[u]]$ such that 
	$$
	e_{\theta} + \sum_{ \theta' \neq \theta} \alpha_{\theta'} e_{\theta'} \in M 
	$$
\end{remark}

\subsection{Low dimensional cases} For weights in the Fontaine--Laffaille range $[0,p-1]$ the inclusion $M \subset f_*N$ from Proposition~\ref{approximate} is an equality:
\begin{lemma}\label{FLrange}
	Suppose $M \in \operatorname{Mod}^{\operatorname{SD}}_k(\mathcal{O})$ is irreducible and $\operatorname{Weight}_\tau(M) \subset [0,p-1]$ for every $\tau$. Then $M = f_*N$ for a rank one $N \in \operatorname{Mod}_l^{\operatorname{SD}}(\mathcal{O})$
\end{lemma}
\begin{proof}
	With notation as in Proposition~\ref{approximate}, if $\delta_{\theta \circ \varphi} =1$ and $\delta_{\theta} = 0$ then $r_{\theta} + p \in \operatorname{Weight}_{\theta|_k}(M) \subset [0,p-1]$ which is impossible. Therefore either all $\delta_\theta = 0$ (in which case $M = f_*N$) or all $\delta_\theta = 1$. In the later case, since $r_{\theta} + p-1 \in [0,p-1]$ it follows that $r_{\theta} = 0$ for all $\theta$, contradicting the irreducibility of $T(f_*N)$.
\end{proof}

On the other hand the previous lemma does not hold for every irreducible $M \in \operatorname{Mod}^{\operatorname{SD}}_k(\mathcal{O})$. In \cite[\S6.4]{B18} an example is given when $K = \mathbb{Q}_p$ and $\operatorname{dim}_{\mathbb{F}_p} T(M) = 5$. We conclude this section by showing this is the lowest dimensional counterexample. We do this by an explicit calculation.
\begin{proposition}\label{Qpcase}
	Suppose $k = \mathbb{F}_p$. Let $M \in \operatorname{Mod}^{\operatorname{SD}}_{k}(\mathcal{O})$ be irreducible with $\operatorname{dim}_{\mathbb{F}} T(M) \leq 4$. After possibly enlarging $\mathbb{F}$, $M \cong f_*N$ for a rank one $N \in \operatorname{Mod}^{\operatorname{SD}}_l(\mathcal{O})$.
\end{proposition}
\begin{proof}
	We only consider the case $\operatorname{dim}_{\mathbb{F}}T(M) =4$, the $2$ and $3$-dimensional cases being much easier. We put ourselves in the situation of Proposition~\ref{approximate}. 
	
	First suppose $e_{\theta} \not\in M$ for all $\theta$ (i.e. $\delta_\theta =1$). Then $r_\theta + p -1 \in [0,p]$ so $r_\theta \in [0,1]$. Fix $\theta \in \operatorname{Hom}_{\mathbb{F}_p}(l,\mathbb{F})$ and set $r := (r_{\theta \circ \varphi^3},r_{\theta \circ \varphi^2},r_{\theta \circ \varphi},r_\theta)$. By replacing $\theta$ with $\theta \circ \varphi^i$ we may assume $r$ is one of the following:
	$$
	(0,0,0,0), \quad (1,1,1,1), \quad  (1,0,1,0), \quad (0,0,1,1)
	$$
	Note that in the first three cases $f_*N$ is not irreducible (consider the sub-Breuil--Kisin module generated by $e_{\theta} + e_{\theta \circ \varphi}^2$ and $e_{\theta \circ \varphi} + e_{\theta \circ \varphi^3}$), so only the last case is possible. Since $u(f_*N) \subset M$, Remark~\ref{surjec} implies there exist $\alpha, \beta ,\gamma \in \mathbb{F}$ not all zero such that $\gamma e_{\theta \circ \varphi^3} + \beta e_{\theta \circ \varphi^2} + \alpha e_{\theta \circ \varphi} + e_\theta \in M$. Applying $\varphi$ gives that $x(\gamma e_{\theta \circ \varphi^2} + \beta ue_{\theta \circ \varphi} + \alpha ue_{\theta} + e_{\theta \circ \varphi^3}) \in M$ and so $e_{\theta \circ \varphi^3} + \gamma e_{\theta \circ \varphi^2} \in M$. Applying $\varphi$ again gives $e_{\theta \circ \varphi^2} + \gamma u e_{\theta \circ \varphi} \in M$, so $e_{\theta \circ \varphi^2} \in M$, a contradiction.
	
	Now suppose $e_{\theta} \in M$ and $e_{\theta \circ \varphi} \not\in M$ (i.e. $\delta_{\theta} = 0, \delta_{\theta \circ \varphi} =1$). The requirement of Remark~\ref{surjec} restricts the possible image of $M/u(f_*N) \hookrightarrow (f_*N)/u(f_*N)$; it must be an $\mathbb{F}$-subspace $V \subset (f_*N)/u(f_*N)$ not containing $e_{\theta \circ \varphi}$ and such that each projection $V \rightarrow (f_*N)/u(f_*N) \xrightarrow{\operatorname{proj}} \mathbb{F}e_{\theta \circ \varphi^i}$ is surjective. One easily checks any $V$ satisfying these properties is one of the following, for some $\alpha,\beta \in \mathbb{F}^\times$:
	\begin{enumerate}
		\item $V = \mathbb{F}( e_{\theta \circ \varphi^3} + \alpha e_{\theta \circ \varphi^2} + \beta e_{\theta \circ \varphi}) + \mathbb{F} e_{\theta}$,
		\item $V = \mathbb{F}( e_{\theta \circ \varphi^3} + \alpha e_{\theta \circ \varphi}) + \mathbb{F}( e_{\theta \circ \varphi^2} + \beta e_{\theta \circ \varphi}) + \mathbb{F}e_\theta$,
		\item $V = \mathbb{F}e_{\theta \circ \varphi^3} + \mathbb{F}( e_{\theta \circ \varphi^2} + \alpha e_{\theta \circ \varphi}) + \mathbb{F} e_\theta$,
		\item $V = \mathbb{F}( e_{\theta \circ \varphi^3} + \alpha e_{\theta \circ \varphi}) + \mathbb{F}e_{\theta \circ \varphi^2} + \mathbb{F} e_{\theta}$.
	\end{enumerate}
	In each of (1)-(4) we write $e_{\theta \circ \varphi^i}$ when we mean the image of $e_{\theta \circ \varphi^i}$ in $(f_*N)/u(f_*N)$. To ease notation we now write $e_i$ in place of $e_{\theta \circ \varphi^i}$, likewise we write $r_{i} = r_{\theta \circ \varphi^i}$ and $\delta_i = \delta_{\theta \circ \varphi^i}$. We now show each of (1)-(4) cannot occur.
	
	\begin{itemize}
		\item In case (1) the elements $e_{3} + \alpha e_{2} + \beta e_{1}$, $ue_{2}$, $u e_{1}$ and $e_{0}$ of $M$ form an $\mathbb{F}[[u]]$-basis. Note then that $\delta_0 = 0, \delta_1 = \delta_2 = \delta_3 = 1$ and so $r_{0} = 0, r_{1}, r_2 \in [0,1]$ and $r_3 > 0$. Further, since $\varphi(e_3 + \alpha e_2 + \beta e_1) = x(u^{r_2}e_2+ u^{r_1}\alpha e_1 + \beta e_0) \in M$ we must have $r_1= r_2 = 1$. Therefore, with respect to the above basis, the matrix of $\varphi_M$ is 
		{\small
		$$
		\hspace*{-1.3cm}
		x\begin{pmatrix}
		0 & 0 & 0 & u^{r_{3}} \\ 1 & 0 & 0 & -\alpha u^{r_3-1} \\ \alpha & u^p & 0 & -\beta u^{r_3-1} \\ \beta & 0 & u^p & 0
		\end{pmatrix} = x\begin{pmatrix}
		0 & 0 & 0 & \frac{1}{\beta} \\ 0 & -\frac{\beta}{\alpha} & 1 & \frac{\beta - \alpha^2}{\alpha \beta} \\ \alpha \beta & \beta u & 0 & -u \\ 0 & -\frac{1}{\alpha} & 0 & \frac{1}{\alpha \beta}
		\end{pmatrix}^{-1} \begin{pmatrix}
		1 & 0 & 0 & 0 \\ 0 & u^p & 0 & 0 \\ 0 & 0 & u^{p+1} & 0 \\ 0 & 0 & 0 & u^{r_3-1}
		\end{pmatrix}\begin{pmatrix}
		1 & 0 & \frac{u^p}{\beta} & 0 \\ 0 & 1 & \frac{-\alpha^2 +\beta}{\alpha\beta} & 0 \\ 0 & 0 & -1 & 0 \\ 0 & 0 & \frac{u^{p+r_3-1}}{\alpha\beta} & 1
		\end{pmatrix}^{-1}
		$$	
		} 
		Remark~\ref{explicit} implies $p+1$ is a weight of $M$, so $M$ is not strongly divisible.
		\item In case (2) the elements $e_3+\alpha e_1,e_2+\beta e_1, ue_1,e_0$ form a basis of $M$ over $\mathbb{F}[[u]]$. In this case $\delta_3 = \delta_2 = \delta_1 = 1$ and $\delta_0 = 0$ so again $r_0 = 0$, $r_{1},r_2 \in [0,1]$, and $r_3 > 0$. Since $\varphi(e_3+\alpha e_1) \in M$ we deduce $r_2 =1$. Similarly $r_1 = 1$ because $\varphi(e_2+\beta e_1) \in M$. With respect to this basis the matrix of $\varphi_M$ is given by 
		\small{
		$$
		\hspace*{-1.7cm}
		\begin{aligned}
		x\begin{pmatrix}
		0 & 0 & 0 & u^{r_3} \\ u & 0 & 0 & 0 \\ -\beta & 1 & 0 & -\alpha u^{r_3-1} \\ \alpha & \beta & u^p & 0 
		\end{pmatrix} &=x
		\begin{pmatrix}
		0 & 0 & 0 & \frac{1}{\alpha} \\ 0 & 0 & \frac{\alpha}{\alpha+ \beta^2} & \frac{\beta}{\alpha+\beta^2} \\ -\alpha \beta & -(\alpha+\beta^2) & -\beta u & u \\ 1 & 0 & 0 & 0
		\end{pmatrix}^{-1} \begin{pmatrix}
		1 & 0 & 0 & 0 \\ 0 & 1 & 0 & 0 \\ 0 & 0 & u^{p+1} & 0 \\ 0 & 0 & 0 & u^{r_3}
		\end{pmatrix} \begin{pmatrix}
		1 & -\frac{\beta}{\alpha} & -\frac{u^p}{\alpha+ \beta^2} & -\frac{\alpha\beta u^{r_3-1}}{\alpha+ \beta^2} \\ 0 & 1 & -\frac{\beta u^p}{\alpha+ \beta^2} & \frac{\alpha^2 u^{r_3-1}}{\alpha+\beta^2} \\ 0 & 0& 1 & 0 \\ 0 & 0 & 0 & 1
		\end{pmatrix}^{-1} 
		\end{aligned}
		$$
		}
		if $\alpha+\beta^2 \neq 0$, or
		\small{
		$$
		\hspace*{-1.3cm}
		\begin{aligned}
		x\begin{pmatrix}
		0 & 0 & 0 & u^{r_3} \\ u & 0 & 0 & 0 \\ -\beta & 1 & 0 & -\alpha u^{r_3-1} \\ \alpha & \beta & u^p & 0 
		\end{pmatrix} &=x \begin{pmatrix}
		0 & 0 & 0 & -\frac{1}{\beta^2} \\ 0 & \beta & 0 & \frac{u}{\beta} \\ \beta^3 & 0 & -\beta u & u \\ 0 & 0 & \frac{1}{\beta^2} & -\frac{1}{\beta^3}
		\end{pmatrix}^{-1} \begin{pmatrix}
		1 & 0 & 0 & 0 \\ 0 & u & 0 & 0 \\ 0 & 0 & u^{p+1} & 0 \\ 0 & 0 & 0 & u^{r_3-1}
		\end{pmatrix} \begin{pmatrix}
		1 & \frac{1}{\beta} & 0 & 0 \\ 0 & 1 & -\frac{u^p}{\beta} & 0 \\ 0 & 0 & 1 & 0 \\ 0 & 0 & \frac{u^{p-r_3+1}}{\beta^3} & 1
		\end{pmatrix}^{-1}
		\end{aligned}
		$$	
		}
		if $\alpha+\beta^2 = 0$. Again $p+1$ is a weight of $M$ so $M$ is not strongly divisible.
		\item In case (3) the elements $e_3,e_2 + \alpha e_1,ue_1,e_0$ form a $\mathbb{F}[[u]]$-basis of $M$. Since $\delta_0 = 0$ and $\delta_1 = \delta_2 = 1$ we have $r_0 = 0$, $r_1 \in [0,1]$ and $r_2>0$. Further, since $\varphi(e_2+ \alpha e_1) \in M$ we must have $r_1 = 1$. With respect to this basis the matrix of $\varphi_M$ is given by 
		\small{
		$$
		\hspace*{-1.3cm}
		x\begin{pmatrix}
		0 & 0 & 0 & u^{r_3} \\ u^{r_2} & 0 & 0 & 0 \\ -\alpha u^{r_2-1} & 1 & 0 & 0 \\ 0 & \alpha & u^p & 0
		\end{pmatrix} = x
		\begin{pmatrix}
		0 & 0 & -\frac{1}{\alpha} & \frac{1}{\alpha^2} \\ 0 & 0 & 0 & \frac{1}{\alpha} \\ 0 & -\alpha^2 & -\alpha u & u \\ 1 & 0 & 0 & 0 
		\end{pmatrix}^{-1}
		\begin{pmatrix}
		u^{r_2-1} & 0 & 0 & 0 \\ 0 & 1 & 0 & 0 \\ 0 & 0 & u^{p+1} & 0 \\ 0 & 0 & 0 & u^{r_3}
		\end{pmatrix}
		\begin{pmatrix}
		1 & 0 & -\frac{u^{p-r_2+1}}{\alpha^2} & 0 \\ 0 & 1 & -\frac{u^p}{\alpha} & 0 \\ 0 & 0 & 1 & 0 \\ 0 & 0 & 0 & 1
		\end{pmatrix}^{-1}
		$$	
		}
		Once more $M$ is not strongly divisible.
		\item In case (4) the elements $e_3+\alpha e_1,e_2,ue_1,e_0$ form an $\mathbb{F}[[u]]$-basis of $M$. Since $\delta_0 = \delta_2 = 0$ and $\delta_1 = \delta_3 = 1$ we have $r_0 = r_2 = 0$ and $r_1,r_3 > 0$. With respect to this basis the matrix of $\varphi_M$ is given by
		\small{
		$$
		\hspace*{-1.3cm}
		x\begin{pmatrix}
		0 & 0 & 0 & u^{r_3} \\ 1 & 0 & 0 & 0 \\ 0 & u^{r_1-1} & 0 & -\alpha u^{r_3-1} \\ \alpha & 0 & u^p & 0
		\end{pmatrix} = x\begin{pmatrix}
		0 & 1 & 0 & 0 \\ 0 & 0 & 1 & 0 \\ 0 & -a & 0 & 1 \\ 1 & 0 & 0 & 0
		\end{pmatrix}^{-1} \begin{pmatrix}
		1 & 0 & 0 & 0 \\ 0 & u^{r_1-1} & 0 & 0 \\ 0 & 0 & u^p & 0 \\ 0 & 0 & 0 & u^{r_3} 
		\end{pmatrix} \begin{pmatrix}
		1 & 0 & 0 & 0 \\ 0 & 1 & 0 & \alpha u^{r_3-r_1} \\ 0 & 0 & 1 & 0 \\ 0 & 0 & 0 & 1
		\end{pmatrix}^{-1}
		$$	
		}If $M$ is strongly divisible then the rightmost matrix must lie in $\operatorname{GL}_n(\mathbb{F}[[u^p]])$. Therefore $p \mid r_3 - r_1$; as both lie in $[1,p]$ we must have $r_3 =r_1$. However then $f_*N$ is not irreducible (consider the sub-Breuil--Kisin module generated by $e_3+e_1$ and $e_2+ e_0$).
\end{itemize}

We conclude that $e_\theta \in M$ for every $\theta$. In other words, $M = f_*N$.
\end{proof}
\subsection{Crystalline liftings} We deduce the following.
\begin{corollary}\label{liftingcoro}
	Assume $K = K_0$. Let $M \in \operatorname{Mod}^{\operatorname{SD}}_k(\mathcal{O})$ be irreducible and assume that one of the following holds.
	\begin{enumerate}
		\item $M$ has rank one over $k[[u]] \otimes_{\mathbb{F}_p} \mathbb{F}$.
		\item $\operatorname{Weight}_\tau(M) \subset [0,p-1]$ for every $\tau \in \operatorname{Hom}_{\mathbb{F}_p}(k,\mathbb{F})$.
		\item $k = \mathbb{F}_p$ and $\operatorname{dim}_{\mathbb{F}} T(M) \leq 4$.
	\end{enumerate}
	Then, after possibly extending $\mathbb{F}$, there exists a crystalline $\mathcal{O}$-lattice $T$ such that $M(T) \otimes_{\mathcal{O}} \mathbb{F} \cong M$ and such that $\operatorname{HT}_\tau(T) = \operatorname{Weight}_\tau(M)$. Further, $T$ is induced over an unramified extension from a crystalline character.
\end{corollary}
\begin{proof}
	If $M$ is an in (1) then this follows from  Proposition~\ref{crystliftrk1}. If $M$ is as in (1) or (2) then by Lemma~\ref{FLrange} and Proposition~\ref{Qpcase}, after possibly enlarging $\mathbb{F}$, we have that $M \cong f_*N$ where $N$ is a rank one Breuil--Kisin module over $L$. Using Proposition~\ref{crystliftrk1} again, there exists a crystalline character $\chi\colon G_L \rightarrow \mathcal{O}^\times$ such that $M(\chi) \otimes_{\mathcal{O}} \mathbb{F}$ and $\operatorname{HT}_\theta(\chi) = \operatorname{Weight}_\theta(N)$. Since $L/K$ is unramified, $\operatorname{Ind}_L^K \chi$ is again crystalline and 
	$$
	\operatorname{HT}_\tau(\operatorname{Ind}_L^K \chi) = \bigcup_{\theta|_k = \tau} \operatorname{HT}_\theta(\chi) = \bigcup_{\theta|_k = \tau} \operatorname{Weight}_\theta(N) = \operatorname{Weight}_\tau(M)
	$$
	cf.~\cite[Corollary 7.1.2]{GHS}. Since $T(f_*M(\chi)) = \operatorname{Ind}_{L_\infty}^{K_\infty} \chi$ which equals the restriction to $G_{K_\infty}$ of $\operatorname{Ind}_L^K \chi$ we deduce that $f_*M(\chi) = M(\operatorname{Ind}_L^K\chi)$, and so
	$$
	M(\operatorname{Ind}_L^K \chi) \otimes_{\mathcal{O}} \mathbb{F} = f_*\Big(M(\chi) \otimes_{\mathcal{O}} \mathbb{F}\Big) \cong \overline{M}
	$$ 
	Therefore we can take $T = \operatorname{Ind}_L^K \chi$.

\end{proof}
\section{Potentially diagonalisable lifts}\label{section7} We use the notation of potential diagonalisability as described in \cite{BLGGT} (the definition is given in the paragraph before \cite[Lemma 1.4.1]{BLGGT}). In practice we shall only use the following elementary observations regarding potential diagonalisability.
\begin{itemize} 
	\item If a crystalline $\mathcal{O}$-lattice $T$ admits a $G_K$-stable filtration $F^iT$ whose graded pieces are free $\mathcal{O}$-modules then $T$ is potentially diagonalisable if and only if $\operatorname{gr}(T) : = \bigoplus \operatorname{gr}^i(T)$ is potentially diagonalisable, cf. property (7) in the list preceding \cite[Lemma 1.4.1]{BLGGT}.
	\item If a crystalline $\mathcal{O}$-lattice $T$ is induced from a character then $T$ is potentially diagonalisable. This is because if $T$ is induced over an extension $L/K$ then $T|_{G_L}$ admits a $G_L$-stable filtration whose graded pieces are characters and so $T|_{G_L}$ is potentially diagonalisable by the previous bullet point (cf. \cite[Lemma 1.4.3]{BLGGT}). Thus $T$ is potentially diagonalisable also.
\end{itemize} 

\subsection{Obvious lifts} 
The previous two bullet points imply that, with the following definition, obvious\footnote{Our terminology comes from \cite[Subsection 7.1]{GHS} where the notion of an obvious weight is introduced.} crystalline $\mathcal{O}$-lattices are potentially diagonalisable.
\begin{definition}
	A crystalline $\mathcal{O}$-lattice is obvious if it admits a filtration $F^iT$ by $G_K$-stable submodules such that each graded piece $\operatorname{gr}^i(T)$ is irreducible after inverting $p$ and induced from a crystalline character over an unramified extension of $K$.
\end{definition}

\begin{theorem}\label{thm}
	Assume $K = K_0$ and $\pi$ is chosen so that $K_\infty \cap K(\mu_{p^\infty}) = K$. Let $\overline{M} \in \operatorname{Mod}^{\operatorname{SD}}_k(\mathcal{O})$ and assume $T(\overline{M})$ is cyclotomic-free (as in Definition~\ref{cycfreeGKinfty}) and that one of the following conditions is satisfied.
	\begin{itemize}
		\item Every irreducible subquotient of $\overline{M}$ is of rank one.
		\item $\operatorname{Weight}_\tau(\overline{M}) \subset [0,p-1]$ for every $\tau \in \operatorname{Hom}_{\mathbb{F}_p}(k,\mathbb{F})$.
		\item $k = \mathbb{F}_p$ and every irreducible subquotient of $\overline{M}$ has rank $\leq 4$.
	\end{itemize}
	Then, after possibly enlarging $\mathbb{F}$, there exists an obvious crystalline $\mathcal{O}$-lattice $T$ such that $M(T) \otimes_{\mathcal{O}} \mathbb{F} \cong \overline{M}$.
\end{theorem}
\begin{proof}
	We argue by induction on the length of $T(\overline{M})$. If $T(\overline{M})$ is irreducible then $\overline{M}$ is irreducible and the theorem follows from Corollary~\ref{liftingcoro}. If $T(\overline{M})$ has length $>1$ then there is an exact sequence $0 \rightarrow \overline{T}_1 \rightarrow T(\overline{M}) \rightarrow\overline{T}_2 \rightarrow 0$ with neither $\overline{T}_1,\overline{T}_2 = 0$. Since $T(\overline{M})$ is cyclotomic-free we can assume $\overline{T}_1$ is cyclotomic-free, $\overline{T}_2$ is irreducible, and $\overline{T}_2 \otimes \overline{\mathbb{F}}(1)$ is not a Jordan--Holder factor of $\overline{T}_1$. Lemma~\ref{exactly} gives an exact sequence $0 \rightarrow \overline{M}_2 \rightarrow \overline{M} \rightarrow \overline{M}_1 \rightarrow 0$ in $\operatorname{Mod}^{\operatorname{BK}}_k(\mathcal{O})$ with $T(\overline{M}_i) = T_i$. Since $\overline{M} \in \operatorname{Mod}^{\operatorname{SD}}_k(\mathcal{O})$ so are the $\overline{M}_i$ by Proposition~\ref{SDexact}. Proposition~\ref{SDexact} also implies $\operatorname{Weight}_\tau(\overline{M}_1) \bigcup \operatorname{Weight}_\tau(\overline{M}_2) = \operatorname{Weight}_\tau(\overline{M})$. Thus both $\overline{M}_i$ satisfy the conditions of the theorem. Our inductive hypothesis provides obvious crystalline $\mathcal{O}$-lattices $T_i$ such that $M(T_i) \otimes_{\mathcal{O}} \mathbb{F} = \overline{M}_i$. This puts us in the situation of Proposition~\ref{liftingextThm}, and this proposition provides us with a crystalline $\mathcal{O}$-lattice $T$ as desired.
\end{proof}
As in the introduction, a crystalline lift of a representation $\overline{\rho}\colon G_K \rightarrow \operatorname{GL}_n(\overline{\mathbb{F}}_p)$ is a crystalline representation $\rho\colon G_K \rightarrow \operatorname{GL}_n(\overline{\mathbb{Z}}_p)$ such that $\rho \otimes_{\overline{\mathbb{Z}}_p} \overline{\mathbb{F}}_p \cong \overline{\rho}$. We say $\rho$ is an obvious crystalline lift if $\rho$ is an obvious crystalline $\mathcal{O}$-lattice.
\begin{corollary}
	Suppose $K = K_0$. Let $\overline{\rho}\colon G_K \rightarrow \operatorname{GL}_n(\overline{\mathbb{F}}_p)$ be continuous, cyclotomic-free, and with Jordan--Holder factors all characters. Then $\overline{\rho}$ admits a  crystalline lift $\rho$ with $\operatorname{HT}_\tau(\rho) \subset [0,p]$ for each $\tau$, if and only if $\overline{\rho}$ admits an obvious crystalline lift $\rho'$ with $\operatorname{HT}_\tau(\rho) = \operatorname{HT}_\tau(\rho')$.
\end{corollary}
\begin{proof}
	Choose $\pi$ so that $K_\infty \cap K(\mu_{p^\infty}) = K$. Suppose such a $\rho$ exists. Choosing our coefficient field $E$ sufficiently large we may suppose $\rho$ factors through $\operatorname{GL}_n(\mathcal{O})$. Theorem~\ref{GLS} implies $M(\rho) \in \operatorname{Mod}^{\operatorname{SD}}_k(\mathcal{O})$ and $\operatorname{Weight}_\tau(M) = \operatorname{HT}_\tau(\rho)$. Put $\overline{M} = M(\rho) \otimes_{\mathcal{O}} \mathbb{F}$. Then $T(\overline{M}) \cong \overline{\rho}$ as $G_{K_\infty}$-representations, so $T(\overline{M})$ is cyclotomic-free as a $G_{K_\infty}$-representation. Theorem~\ref{thm} implies there exists a potentially diagonalisable $\rho'$ with $M(\rho') \otimes_{\mathcal{O}} \mathbb{F} \cong \overline{M}$ and with $\operatorname{HT}_\tau(\rho') = \operatorname{Weight}_\tau(\overline{M})$. Thus $\rho' \otimes_{\mathcal{O}} \mathbb{F}  \cong \overline{\rho}$ as $G_{K_\infty}$-representations. By Theorem~\ref{ff} they are isomorphic as $G_K$-representations and we are done.
\end{proof}
If we assume $K = \mathbb{Q}_p$ we can also prove:
\begin{corollary}
	Let $\overline{\rho}\colon G_{\mathbb{Q}_p} \rightarrow \operatorname{GL}_n(\overline{\mathbb{F}}_p)$ be continuous and cyclotomic-free, with $n \leq 5$. Then $\overline{\rho}$ admits a  crystalline lift $\rho$ with $\operatorname{HT}(\rho) \subset [0,p]$ for each $\tau$, if and only if $\overline{\rho}$ admits an obvious crystalline lift $\rho'$ with $\operatorname{HT}(\rho) = \operatorname{HT}(\rho')$.
\end{corollary}
\begin{proof}
	When every Jordan--Holder factor of $\overline{\rho}$ has dimension $\leq 4$ this follows from Theorem~\ref{thm} by the argument used in the above corollary. If $\overline{\rho}$ is irreducible this fact follows from the main result of \cite{B18}, cf. Theorem~7.2.1.
\end{proof}
\bibliography{/home/robin/Dropbox/Maths/biblio}

\begin{bibdiv}
\begin{biblist}

\bib{BLGG}{article}{
      author={Barnet-Lamb, Thomas},
      author={Gee, Toby},
      author={Geraghty, David},
       title={Serre weights for {$U(n)$}},
        date={2018},
     journal={J. Reine Angew. Math.},
      volume={735},
       pages={199\ndash 224},
}

\bib{BLGGT}{article}{
      author={Barnet-Lamb, Thomas},
      author={Gee, Toby},
      author={Geraghty, David},
      author={Taylor, Richard},
       title={Potential automorphy and change of weight},
        date={2014},
     journal={Ann. of Math. (2)},
      volume={179},
      number={2},
       pages={501\ndash 609},
}

\bib{B18}{article}{
      author={Bartlett, Robin},
       title={Inertial and {H}odge--{T}ate weights of crystalline
  representations},
        date={2018},
     journal={Preprint, arXiv:1811.10260},
}

\bib{BG18}{article}{
      author={Bellovin, Rebecca},
      author={Gee, Toby},
       title={G-valued local deformation rings and global lifts},
        date={2018},
     journal={Preprint, arXiv:1708.04885},
}

\bib{BMS}{article}{
      author={Bhatt, Bhargav},
      author={Morrow, Matthew},
      author={Scholze, Peter},
       title={Integral {$p$}-adic {H}odge theory},
        date={2016},
     journal={Preprint, arXiv:1602.03148},
}

\bib{Fon00}{incollection}{
      author={Fontaine, Jean-Marc},
       title={Repr\'esentations {$p$}-adiques des corps locaux. {I}},
        date={1990},
   booktitle={The {G}rothendieck {F}estschrift, {V}ol.\ {II}},
      series={Progr. Math.},
      volume={87},
   publisher={Birkh\"auser Boston, Boston, MA},
       pages={249\ndash 309},
}

\bib{Fon94b}{article}{
      author={Fontaine, Jean-Marc},
       title={Repr\'esentations {$p$}-adiques semi-stables},
        date={1994},
     journal={Ast\'erisque},
      number={223},
        note={With an appendix by Pierre Colmez, P\'eriodes $p$-adiques
  (Bures-sur-Yvette, 1988)},
}

\bib{Gao17a}{article}{
      author={Gao, Hui},
       title={Crystalline liftings and weight part of {S}erre's conjecture},
        date={2017},
     journal={Israel J. Math},
      number={221},
       pages={no. 1, 117\ndash 164},
}

\bib{Gao17b}{article}{
      author={Gao, Hui},
       title={A {note} on crystalline liftings in the $\mathbb{Q}_p$ case},
        date={2017},
     journal={Bulletin de la SMF (to appear)},
}

\bib{GT14}{article}{
      author={Gao, Hui},
      author={Liu, Tong},
       title={A note on potential diagonalizability of crystalline
  representations},
        date={2014},
     journal={Math. Ann.},
      volume={360},
      number={1-2},
       pages={481\ndash 487},
}

\bib{GHS}{article}{
      author={Gee, Toby},
      author={Herzig, Florian},
      author={Savitt, David},
       title={General {S}erre weight conjectures},
        date={2018},
     journal={Journal of the European Math Society},
}

\bib{GLS}{article}{
      author={Gee, Toby},
      author={Liu, Tong},
      author={Savitt, David},
       title={The {B}uzzard-{D}iamond-{J}arvis conjecture for unitary groups},
        date={2014},
     journal={J. Amer. Math. Soc.},
      volume={27},
      number={2},
       pages={389\ndash 435},
}

\bib{GLS15}{article}{
      author={Gee, Toby},
      author={Liu, Tong},
      author={Savitt, David},
       title={The weight part of {S}erre's conjecture for
  {$\operatorname{GL}(2)$}},
        date={2015},
     journal={Forum Math. Pi},
      volume={3},
       pages={e2, 52},
}

\bib{Kis06}{incollection}{
      author={Kisin, Mark},
       title={Crystalline representations and {$F$}-crystals},
        date={2006},
   booktitle={Algebraic geometry and number theory},
      series={Progr. Math.},
      volume={253},
   publisher={Birkh\"auser Boston, Boston, MA},
       pages={459\ndash 496},
}

\bib{Nek93}{incollection}{
      author={Nekov\'ar, Jan},
       title={On {$p$}-adic height pairings},
        date={1993},
   booktitle={S\'eminaire de {T}h\'eorie des {N}ombres, {P}aris, 1990--91},
      series={Progr. Math.},
      volume={108},
   publisher={Birkh\"auser Boston, Boston, MA},
       pages={127\ndash 202},
}

\bib{Serre72}{article}{
      author={Serre, Jean-Pierre},
       title={Propri\'{e}t\'{e}s galoisiennes des points d'ordre fini des
  courbes elliptiques},
        date={1972},
     journal={Invent. Math.},
      volume={15},
      number={4},
       pages={259\ndash 331},
}

\bib{Serre}{book}{
      author={Serre, Jean-Pierre},
       title={Galois cohomology},
     edition={English},
      series={Springer Monographs in Mathematics},
   publisher={Springer-Verlag, Berlin},
        date={2002},
        note={Translated from the French by Patrick Ion and revised by the
  author},
}

\bib{Tate76}{article}{
      author={Tate, John},
       title={Relations between {K}2 and {G}alois cohomology},
        date={1976},
     journal={Inventiones mathematicae},
      volume={36},
       pages={257\ndash 274},
}

\bib{Wang17}{article}{
      author={Wang, Xiyuan},
       title={{Weight elimination in two dimensions when $p=2$}},
        date={2017},
      eprint={arXiv:1711.09035},
}

\end{biblist}
\end{bibdiv}
\end{document}